\documentclass[12pt,a4paper]{amsart}

\calclayout

\newcommand{\+}{\nobreakdash-}
\renewcommand{\.}{\mskip .5\thinmuskip\relax}
\renewcommand{\:}{\colon}
\renewcommand{\;}{,\medspace}

\let\le\undefined
\DeclareMathSymbol{\le}{\mathrel}{AMSa}{"36}         %\leqslant
\let\ge\undefined
\DeclareMathSymbol{\ge}{\mathrel}{AMSa}{"3E}         %\geqslant

\DeclareMathOperator{\Spec}{Spec}
\DeclareMathOperator{\Hom}{Hom}
\DeclareMathOperator{\Ext}{Ext}
\DeclareMathOperator{\id}{id}

\newcommand{\rarrow}{\longrightarrow}
\newcommand{\ot}{\otimes}
\newcommand{\lrarrow}{\.\relbar\joinrel\relbar\joinrel\rightarrow\.}
\newcommand{\bu}{{\text{\smaller\smaller$\scriptstyle\bullet$}}}
\renewcommand{\d}{\partial}

\newcommand{\E}{{\mathcal E}}
\newcommand{\F}{{\mathcal F}}
\newcommand{\D}{{\mathcal D}}
\newcommand{\M}{{\mathcal M}}
\newcommand{\A}{{\mathcal A}}
\newcommand{\T}{{\mathcal T}}
\newcommand{\U}{{\mathcal U}}
\newcommand{\Z}{{\mathbb Z}}
\newcommand{\Q}{{\mathbb Q}}
\newcommand{\R}{{\mathbb R}}
\renewcommand{\H}{{\mathbb H}}

\newcommand{\Et}{\vphantom{\acute E}
   \smash{\mathit{\acute Et}}\vphantom{E}}
\newcommand{\Nis}{{\mathit Nis}}
\newcommand{\Zar}{{\mathit Zar}}
\newcommand{\gr}{{\mathrm gr}}
\newcommand{\cc}{{cc}}

\newcommand{\usm}[1]{\underline{\smash{#1}\!\.}\.}

\newcommand{\Section}[1]{\bigskip\section{#1}\medskip}

\theoremstyle{plain}
\newtheorem{thm}{Theorem}[section]

\newtheorem{lem}[thm]{Lemma}
\newtheorem{prop}[thm]{Proposition}
\newtheorem{cor}[thm]{Corollary}
\theoremstyle{definition}
\newtheorem{rem}[thm]{Remark}

\begin{document}

\title{Artin--Tate motivic sheaves with finite coefficients \\
over an algebraic variety}
\author{Leonid Positselski}

\address{Sector of Algebra and Number Theory, Institute for Information
Transmission Problems, Moscow 127994; and \newline\indent 
Laboratory of Algebraic Geometry, National Research University Higher
School of Economics, Moscow 117312, Russia; and \newline\indent
Mathematics Department, Technion --- Israel Institute of Technology,
Haifa 32000, Israel}

\email{posic@mccme.ru}

\begin{abstract}
 We propose a construction of a tensor exact category $\F_X^m$ of
Artin--Tate motivic sheaves with finite coefficients $\Z/m$ over
an algebraic variety $X$ (over a field $K$ of characteristic prime
to~$m$) in terms of \'etale sheaves of $\Z/m$\+modules over~$X$.
 Among the objects of $\F_X^m$, in addition to the Tate motives
$\Z/m(j)$, there are the cohomological relative motives with compact
support $\M_\cc^m(Y/X)$ of varieties $Y$ quasi-finite over~$X$.
 Exact functors of inverse image with respect to morphisms of algebraic
varieties and direct image with compact supports with respect to
quasi-finite morphisms of varieties $Y\rarrow X$ act on
the exact categories~$\F_X^m$.
 Assuming the existence of triangulated categories of motivic
sheaves $\D\M(X,\allowbreak\Z/m)$ over algebraic varities $X$ over~$K$
and a weak version of the ``six operations'' in these categories,
we identify $\F_X^m$ with the exact subcategory in $\D\M(X,\Z/m)$
consisting of all the iterated extensions of the Tate twists
$\M_\cc^m(Y/X)(j)$ of the motives $\M_\cc^m(Y/X)$.
 An isomorphism of the $\Z/m$\+modules $\Ext$ between the Tate
motives $\Z/m(j)$ in the exact category $\F_X^m$ with the motivic
cohomology modules predicted by the Beilinson--Lichtenbaum \'etale
descent conjecture (recently proven by Voevodsky, Rost, et al.)\ holds
for smooth varieties $X$ over $K$ if and only if the similar
isomorphism holds for Artin--Tate motives over fields containing~$K$.
 When $K$ contains a primitive $m$\+root of unity, the latter
condition is equivalent to a certain Koszulity hypothesis,
as shown in our previous paper~\cite{mat}.
\end{abstract}

\maketitle

\setcounter{tocdepth}{1}
\tableofcontents

\section*{Introduction}
\medskip

 The motivic cohomology of an algebraic variety $X$ over a field $K$
with coefficients in a ring $k=\Z$, \ $\Z/m$, \ $\Q$, \dots\ can be
defined~\cite{Voev-triang,Voev-chow,Voev-simpl} as
$$
 H_\M^i(X,k(j)) = \Hom_{\D\M(X,k)}(k,k(j)[i])
 = \Hom_{\D\M(K,k)}(\M_h^k(X),k(j)[i]),
$$
where $\D\M(K,k)$ is the derived category of mixed motives over $K$
with coefficients in~$k$ and $\M_h^k(X)\in\D\M(K,k)$ is
the homological (covariant) motive of $X$ over~$K$, while $\D\M(X,k)$
is the derived category of motives (motivic sheaves) over $X$
with coefficients in~$k$.
 The motivic cohomology localize in the Zariski topology, so one
has
$$
 H_\M^i(X,k(j)) = \H^i_\Zar(X,\usm{k(j)}),
$$
where $\usm{k(j)}$ are certain complexes of sheaves of $k$\+modules
on the big Zariski site of varieties over~$K$
(defined naturally up to a quasi-isomorphism).
 Furthermore, one has $\usm{k(j)}=k\ot^{\mathbb L}_\Z\usm{\Z(j)}$.
 Here $\H_\Zar$ denotes the hypercohomology of complexes of
sheaves in the Zariski topology and $\ot^{\mathbb L}_\Z$ is
the notation for the left derived functor of tensor product of
sheaves over~$\Z$.

 Now assume that the variety $X$ is smooth and the field $K$ is
perfect of characteristic not dividing a positive integer~$m$.
 Then the motivic cohomology of $X$ with coefficients in $\Z/m$
can be computed by the formula
\begin{equation}  \label{mot-coh-fin-coef}
 H_\M^i(X,\Z/m(j)) = \H^i_\Zar(X,\tau_{\le j}
 \R\pi_*\mu_m^{\ot j}),
\end{equation}
where $\pi\:\Et\rarrow\Zar$ is the natural map between the big \'etale
and Zariski sites of varieties over~$K$ (the direction of
the map of sites being opposite to that of the functor between
the categories of ``open sets''), $\R\pi_*$ denotes the right derived
functor of direct image of sheaves, $\tau_{\le j}$ are the canonical
truncations of complexes of Zariski sheaves, and the cyclotomic
\'etale sheaves $\mu_m^{\ot j}$ are the tensor powers over $\Z/m$ of
the \'etale sheaf~$\mu_m$ of $m$\+roots of unity.
 In fact, the formula~\eqref{mot-coh-fin-coef} is a combination of
two assertions: the Beilinson--Lichtenbaum modified \'etale descent
rule
\begin{equation} \label{BL-etale-descent}
 \usm{\Z(j)} = \tau_{\le j}\mathbb R\pi_*\pi^*\usm{\Z(j)}
 = \tau_{\le j+1}\R\pi_*\pi^*\usm{\Z(j)},
\end{equation}
and (a version of) the Suslin rigidity theorem
\begin{equation} \label{Sus-rigidity}
 \pi^*\usm{\Z/m(j)}=\mu_m^{\ot j}, \qquad j\ge0.
\end{equation}

 The isomorphism~\eqref{Sus-rigidity} goes back to A.~Suslin's
paper~\cite{Sus-acl}; its proof in the form stated above can be
found in~\cite[7.20 and~10.3]{MVW}.
 The rules~(\ref{mot-coh-fin-coef}\+-\ref{BL-etale-descent})
were conjectured by A.~Beilinson~\cite[5.10.D(v-vi)]{Beil}
and S.~Lichtenbaum~\cite{Lich}.
 According to~\cite{SV,GL}, the formula~\eqref{mot-coh-fin-coef}
follows from the Milnor--Bloch--Kato conjecture, the long work on
the proof of which was recently finished by Voevodsky, Rost,
et al.~\cite{Voev-MBK}.
 One can replace the Zariski topology with the Nisnevich topology
in these results~\cite[13.9 and~22.2]{MVW}.

 For singular varieties $X$, the formula~\eqref{mot-coh-fin-coef}
no longer holds, as one can see in the following simple example.
 Let $X$ be the affine line (say, over the field $K$ of complex
numbers) with two different points glued together.
 Then one has $\M_h^\Z(X)=\Z\oplus \Z[1]$, so $H^1_\M(X,\Z/m(0))=
\Z/m$, while $\tau_{\le0}\mathbb R\pi_*\Z/m=\Z/m$ and
$H^1_\Zar(X,\Z/m)=0$.
 In fact, the restriction of $\pi^*\usm{\Z/m(0)}$ to the small
\'etale site of $X$ is still isomorphic to $\Z/m$, since our
first motivic cohomology class of $X$ with coefficients in $\Z/m$
dies in \'etale covers.
 So it must be the modified \'etale descent
rule~\eqref{BL-etale-descent} for the Zariski topology
that breaks down in this case.

 For the above singular curve~$X$, the formula~\eqref{mot-coh-fin-coef}
can still be saved by replacing the Zariski topology with
the Nisnevich topology, as one has $H^1_\Nis(X,\Z/m)=\Z/m$.
 However, for more complicated singularities the Nisnevich
topology is not enough either, and the \emph{cdh} topology is
needed~\cite[Theorem~14.20]{MVW}.
 E.~g., if $Y$ is a normal surface with a point singularity such
that the exceptional fiber of its resolution is
a self-intersecting projective line, then one has
$H^2_\M(Y,\Z/m(0))=\Z/m$, while $\tau_{\le0}\mathbb R\rho_*\Z/m=\Z/m$
and $H^2_\Nis(Y,\Z/m)=0$ ($\rho$~being the natural map $\Et\rarrow
\Nis$ between the big \'etale and Nisnevich sites)
\cite[Exercise~12.32]{MVW}.
 Now it is the formula~\eqref{Sus-rigidity} that breaks down.

 The formula~\eqref{mot-coh-fin-coef} suggests that it might be
possible to construct the derived category of motivic sheaves
$\D\M(X,\Z/m)$ on a variety $X$, or at least some parts of
this category, in terms of the \'etale topology of $X$.
 The aim of this paper is to suggest such a construction for
the triangulated category $\D\M\A\T(X,\Z/m)$ of \emph{Artin--Tate
motivic sheaves} over~$X$.
 This is defined as the full triangulated tensor subcategory of
$\D\M(X,\Z/m)$ generated by Tate motives $\Z/m(j)$ and the compactly
supported relative cohomological motives of varieties $Y$
quasi-finite over~$X$.
 In fact, we construct a $\Z/m$\+linear exact category of filtered
constructible \'etale sheaves of $\Z/m$\+modules $\F^m_X$ over $X$,
whose derived category $\D^b(\F^m_X)$ is similar to
$\D\M\A\T(X,\Z/m)$ ``insofar as no complicated singularities
get involved''.

 We also establish some functoriality properties of the exact
categories $\F_X^m$ with respect to morphisms of varieties
$f\:Y\rarrow X$.
 Namely, we construct exact functors of inverse image
$f^*\:\F_X^m\rarrow\F_Y^m$ for all morphisms~$f$ and exact functors
of direct image with compact supports $f_!\:\F_Y^m\rarrow\F_X^m$
for quasi-finite morphisms~$f$.
 The functors $f^*$ and~$f_!$ are adjoint to each other from
different sides depending on whether $f$~is finite or \'etale.
 For each quasi-finite morphism of varieties $Y\rarrow X$ we define
the relative cohomological motive with compact supports
$\M_\cc^m(Y/X)\in\E_X^m$ of $Y$ over~$X$.
 Here $\E_X^m\subset\F_X^m$ is the full exact subcategory of
\emph{Artin motivic sheaves}.
 For any quasi-finite morphism of smooth varieties $Y\rarrow X$
we have the relative homological motive $\M_h^m(Y/X)\in\D^b(\F_X^m)$.

 Let us first discuss the case when $X=\Spec L$ is the spectrum
of a field.
 In this case, the Tate twists of the motives of the spectra of
finite separable extensions of $L$ generate (using iterated
extensions) an exact subcategory $\M\A\T(L,\Z/m)\subset
\D\M\A\T(L,\Z/m)$ which was computed in~\cite{mat} in terms of
the absolute Galois group $G_L$ of the field~$L$.
 The category $\M\A\T(L,\Z/m)$ is equivalent to the exact category
of finitely filtered discrete $G_L$\+modules over $\Z/m$ whose
successive quotient modules are cyclotomically twisted (finitely
generated) permutational modules.

 The triangulated category $\D\M\A\T(L,\Z/m)$ is equivalent to
the derived category $\D^b\M\A\T(L,\Z/m)$ of the exact category
$\M\A\T(L,\Z/m)$ if and only if the natural maps from
the $\Z/m$\+modules of Yoneda $\Ext$ between the objects of
the exact category to the modules of higher $\Hom$ between
the same objects in the triangulated category are isomorphisms.
 The latter property can be called the \emph{$K(\pi,1)$\+conjecture
for Artin--Tate motives over $L$ with coefficients\/ $\Z/m$}
(cf.\ the discussion of Tate motives with rational coefficients
in~\cite{BK}).
 In the case when $L$ contains a primitive $m$\+root of unity,
this conjecture has been interpreted in~\cite{mat} as a certain
Koszulity hypothesis for the ``big graded ring'' of diagonal
$\Hom$ in $\D\M\A\T(L,\Z/m)$.
 Another name for the $K(\pi,1)$\+conjecture is the \emph{silly
filtration conjecture}.

 This paper purports to explain how to ``globalize''
the $K(\pi,1)$\+conjecture to smooth varieties, in the particular
case of Artin--Tate motives with finite coefficients.
 (Note that we do not know how to globalize the more conventional
$K(\pi,1)$\+conjecture for Tate motives.)
 We proceed in two steps.
 First of all, we construct natural maps from the $\Z/m$\+modules
$\Ext_{\F_X^m}^i(\Z/m,\Z/m(j))$ to the ``naive'' motivic cohomology
modules $\H_\Nis^i(X,\tau_{\le j}\R\rho_*\mu_m^{\ot j})$, and
show that these are isomorphisms (for $X$ and all varieties \'etale
over~$X$) if and only if the $K(\pi,1)$\+conjecture holds for
Artin--Tate motives over the residue fields of
the scheme points of~$X$.

 Secondly, we assume the existence of reasonably well-behaved
triangulated categories of motivic sheaves $\D\M(X,\Z/m)$ over
algebraic varieties $X$ over~$K$ and identify the exact category
$\F_X^m$ with the full subcategory $\M\A\T(X,\Z/m)\subset
\D\M(X,\Z/m)$ generated, using iterated extensions, by the Tate
twists $\M_\cc^m(Y/X)(j)$ of the compactly supported relative
cohomological motives of varieties $Y$ quasi-finite over $X$,
with its induced exact category structure.
 Note that the construction of this fully faithful functor and
equivalence of exact categories does not yet depend on
the $K(\pi,1)$\+conjectures of any kind.

 The $\Z/m$\+modules of higher $\Hom$ in $\D\M(X,\Z/m)$ between
objects of $\M\A\T(X,\allowbreak\Z/m)$ may differ from the modules
$\Ext$ computed in $\F_X^m = \M\A\T(X,\Z/m)$, but, assuming
the $K(\pi,1)$\+conjecture for Artin--Tate motives over fields,
this only happens for singularities-related reasons.
 In particular, in the mentioned assumptions, the groups $\Ext$ in
$\F_X^m$ coincide with groups of higher $\Hom$ in $\D\M\A\T(X,\Z/m)$
when $X$ is a curve, as singularities of curves can be resolved by
finite morphisms.

 In other words, the triangulated subcategory
$\D\M\A\T(X,\Z/m)\subset\D\M(X,\Z/m)$ is equivalent to
the derived category $\D^b\M\A\T(X,\Z/m)$ when $X$ is a curve.
 Let us emphasize that the latter assertion is certainly \emph{not}
true for surfaces (not even for smooth surfaces over algebraically
closed fields, as the category $\M\A\T(X,\Z/m)$ for such a surface
$X$ contains objects related to surfaces $Y$ with bad enough
singularities mapping finitely onto~$X$).
 However, one has $\Ext^i_{\F_X^m}(M,\M_\cc^m(Y/X)(j))\simeq
\Hom_{\D\M\A\T(X,\Z/m)}(M,\M_\cc^m(Y/X)(j)[i])$ for any smooth
variety $Y$ finite over a variety~$X$ and any object
$M\in\F_X^m$.

 We also prove the basic properties of the relative homological
motives $\M_h^m(Y/X)\allowbreak\in\D^b(\F_X^m)$ of smooth varieties $Y$
quasi-finite over a fixed smooth variety $X$ using some of
the conventional assumptions about the triangulated categories of
motivic sheaves $\D\M(X,\Z/m)$.
 In particular, the group $\Hom_{\D^b(\F_X^m)}(\M_h^m(Y/X),\.
\Z/m(j)[i])$ is identified with the group $\Hom_{\D^b(\F_Y^m)}
(\Z/m,\Z/m(j)[i])$ and the motivic cohomology group
$\Hom_{\D\M(Y,\Z/m)}(\Z/m,\Z/m(j)[i])$ (assuming, as above,
the $K(\pi,1)$\+conjecture for Artin--Tate motives over fields).

 Thus our category $\F_X^m$ is proposed as a solution to the problem
of ``constructing explicitly the category of `fine \'etale
$\Z/l^n$\+sheaves'\,'' posed in~\cite[5.10.D(vi)]{Beil}.

 To conclude, let us try to explain why we think it took about
quarter-century to arrive to this solution of a problem of
Beilinson.
 The explanation is that while the formulation of the problem looks
elementary, the solution is not.
 For one thing, it is based on modern concepts in the theory
of motives, such as the Nisnevich and cdh~topologies, which did
not exist at the time of~\cite{Beil}.
 Perhaps even more importantly, it is based on new concepts in
homological algebra.

 The observation that the $K(\pi,1)$/silly filtration conjecture
for Tate motives with finite coefficients is related to the Koszulity
hypothesis for Galois cohomology~\cite{PV} was first reported
in~\cite{ober}.
 At that time, this result was only applicable to Tate rather than
Artin--Tate motives.
 However, the necessity to use Nisnevich (rather than just Zariski)
topology in the main argument of this paper makes it only working in
the generality of Artin--Tate motives.
 Interpreting the silly filtration conjecture for Artin--Tate motives
as a kind of Koszulity hypothesis required generalizing the notion
of Koszulity to nonnegatively graded rings that are not flat over
their zero-degree component.
 That was only achieved in~\cite{mat}.

 Finally, the very construction of the exact category of Artin--Tate
motivic sheaves in this paper is a far-reaching generalization of
the Galois-theoretic construction of the category of Tate motives
with finite coefficients over a field in~\cite{ober}
and~\cite[Section~2]{mat}.
 The proof of the embedding theorem in this paper is also based on
techniques of working with exact categories that appeared in~\cite{mat}.

\subsection*{Acknowledgment}
 I am grateful to Mikhail Bondarko, Tyler Lawson, Dustin Clausen,
Brian Conrad, and Angelo Vistoli for stimulating discussions and
helpful consultations, some of which took place on the web site
\texttt{MathOverflow.net}$\.$.
 I~would like to thank Alexander Kuznetsov and Miles Reid who kindly
answered my questions in Algebraic Geometry.
 Quite separately, I wish to express my gratitude to Vadim Vologodsky,
to whom I owe my understanding of the motivic cohomology of
singular varieties.
 Needless to say, any possible errors are mine.
 The author was partially supported by a grant from P.~Deligne 2004
Balzan prize in mathematics, a Simons Foundation grant, and an RFBR
grant while working on this paper.

\Section{Exact Category of Artin--Tate Motivic Sheaves}
\label{exact-at}

 All schemes in this paper are presumed to be separated.
 We fix an integer $m\ge2$, and a perfect field $K$ of characteristic
not dividing~$m$.
 By an (\emph{algebraic}) \emph{variety} over~$K$ we mean
a scheme of finite type over $\Spec K$, which is not distinguished
from its maximal reduced closed subscheme.
 For the purposes of notation and terminology related to
the dimensions, all smooth varieties are presumed to be
equidimensional.

 Recall that for any Noetherian scheme $X$ the category
$\Et^{m,\infty}_X$ of \'etale sheaves of $\Z/m$\+modules over $X$ is
a locally Noetherian Grothendieck abelian category.
 In other words, $\Et^{m,\infty}_X$ is equivalent to the category of
ind-objects in the abelian category $\Et^m_X$ of constructible \'etale
sheaves of $\Z/m$\+modules over~$X$.
 Here an \'etale sheaf is called \emph{constructible} if it is
generated by a finite set of its sections (over some \'etale schemes
of finite type over $X$).
 This is equivalent to the existence of a finite stratification of $X$
by locally closed subschemes, in restriction to which the sheaf is
locally constant (lisse) with finitely generated stalks.
 An \'etale sheaf of $\Z/m$\+modules is constructible if and only if
it is a Noetherian object in the category of \'etale
sheaves~\cite[Arcata~IV.3]{SGA41/2}.

 For any algebraic variety $X$ over $K$, consider the following
exact category $\E_X^m$.
 The objects of $\E_X^m$ are constructible \'etale sheaves of
$\Z/m$\+modules $N$ over $X$ such that for any scheme point $x\in X$
with the residue field $K_x$ the stalk $N_x$ of $N$ over~$x$,
considered as a discrete module over the absolute Galois group
$G_{K_x}$, is a (finitely generated) permutational module with
coefficients in~$\Z/m$.
 In other words, the $\Z/m$\+module $N_x$ must admit a system of
free generators that is preserved as a set (permuted) by the action
of~$G_{K_x}$.
 As an additive category, $\E_X^m$ is a full subcategory in $\Et_X^m$.
 The exact triples in $\E_X^m$ are the short exact sequences of \'etale
sheaves from $\E_X^m$ for which the related short sequences of
stalks at~$x$ are \emph{split} short exact sequences of
$G_{K_x}$\+modules over $\Z/m$ for all the scheme points $x\in X$.

 The category $\E_X^m$ is suggested as our candidate for the role of
the \emph{exact category of Artin motivic sheaves} over~$X$.
 The larger exact category $\F_X^m$ of \emph{Artin--Tate motivic
sheaves} is constructed on the basis of $\E_X^m$ and $\Et_X^m$ in
the following way.

 The objects of $\F_X^m$ are filtered \'etale sheaves of
$\Z/m$\+modules $(N,F)$ over $X$ with a finite decreasing filtration
by \'etale subsheaves $F^jN$, \ $j\in\Z$, \ $F^jN=N$ for $j\ll0$ and
$0$ for $j\gg0$.
 The successive quotient sheaves $\gr^j_FN=F^jN/F^{j+1}N$ must be
isomorphic to the tensor products over $\Z/m$ of \'etale sheaves
from $\E_X^m$ with the cyclotomic \'etale sheaves $\mu_m^{\ot j}$.
 The latter are the inverse images to $X$ of the \'etale sheaves
over $\Spec K$ corresponding to the cyclotomic representations of
the Galois group $G_K$.
 So one must have $\gr_F^jN\ot_{\Z/m}\mu_m^{\ot -j}\in\E_X^m$.
 The morphisms in $\F_X^m$ are the filtration-preserving morphisms of
\'etale sheaves of $\Z/m$\+modules.
 The exact triples in $\F_X^m$ are the pairs of morphisms with zero
composition whose successive quotients with respect to the filtration
$F$ are exact triples in $\E_X^m$ twisted with $\mu_m^{\ot j}$.
 In other words, it is required that the successive quotients form
short exact sequences of \'etale sheaves whose stalks over every
scheme point $x\in X$ are split exact triples of $G_{K_x}$\+modules.
 
 The exact categories $\E_X$ and $\F_X$ have natural structures of
associative, commutative, and unital tensor categories with exact
functors of tensor product.
 These are given by the tensor products of \'etale sheaves over $\Z/m$
and the tensor products of filtrations.
 The \emph{Tate motive} $\Z/m(j)\in\F_X^m$ is the \'etale sheaf
$\mu_m^{\ot j}$ over $X$, placed in the filtration component~$j$,
so that $F^j\Z/m(j)=\Z/m(j)$ and $F^{j+1}\Z/m(j)=0$.
 The objects $\Z/m(j)$ are invertible in the tensor category~$\F_X^m$;
there are natural isomorphisms $\Z/m(i)\ot \Z/m(j)=\Z/m(i+j)$.
 For any object $N\in\F_X^m$ the tensor product $N\ot\Z/m(j)$ is
denoted by $N(j)$ and called the \emph{Tate twist} of $N$ by~$j$.
 The category $\E_X^m$ is naturally a full exact tensor subcategory
in $\F_X^m$ consisting of all the objects concentrated in
the filtration component~$0$.

 Given an integer $n$ dividing~$m$, and any variety $X$ over $K$,
there is a natural tensor exact functor $\F_X^m\rarrow\F_X^n$,
assigning to a filtered \'etale sheaf of $\Z/m$\+modules $(M,F)$
the \'etale sheaf of $\Z/n$\+modules $M/n\simeq (m/n)M$ with
the induced filtration.

\Section{Inverse Image and Direct Image with Compact Supports}
\label{inverse-direct-section}

 Let $f\:Y\rarrow X$ be a morphism of algebraic varieties over~$K$.
 Then the inverse image functor $f^*\:\Et_X^m\rarrow\Et_Y^m$ takes
$\E_X^m$ to $\E_Y^m$ and is exact as a functor between these exact
categories (and as a functor between the abelian categories
$\Et_X^m$ and $\Et_Y^m$, too).
 The functor $f^*$ is also a tensor functor taking $\mu_m$
to~$\mu_m$, and in particular commuting with the cyclotomic twists.
 Hence there is the induced exact functor $f^*\:\F_X^m\rarrow\F_Y^m$.
 It takes the Tate motives $\Z/m(j)$ over $X$ to the similar Tate
motives $\Z/m(j)$ over $Y$ and commutes with the tensor products in
$\F_X$ and $\F_Y$.

 Let $f\:Y\rarrow X$ be a quasi-finite morphism of varieties
over~$K$.
 Then the functor of direct image with compact supports $f_!\:
\Et_Y^m\rarrow\Et_X^m$ is exact.
 For any \'etale sheaf of $\Z/m$\+modules $N$ over $Y$ the stalk
of the sheaf $f_!\.N$ over a scheme point $x\in X$ is isomorphic to
the direct sum of the direct images of the stalks of $N$ over
the points $y\in Y$ such that $f(y)=x$ with respect to the finite
morphisms of spectra of fields $y\to x$ \cite[Arcata~IV.5]{SGA41/2}.
 On the level of discrete modules over the absolute Galois groups
of the fields $K_x$ and $K_y$, the direct image is the functor of
induction from the open subgroup $G_{K_y}\subset G_{K_x}$.
 
 Hence the functor~$f_!$ takes $\E_Y^m$ to $\E_X^m$ and induces
an exact functor between these exact categories.
 According to the projection formula, the functor $f_!$ also commutes
with the twists by (locally constant constructible) sheaves lifted
from $X$, and in particular from $\Spec K$.
 Thus there is the induced exact functor $f_!\:\F_Y^m\rarrow\F_X^m$,
commuting with the Tate twists.

 For any \'etale morphism of algebraic varieties $f\:Y\rarrow X$,
the functor $f_!\:\F_Y^m\rarrow\F_X^m$ is left adjoint to the functor
$f^*\:\F_X^m\rarrow\F_Y^m$.
 For any finite morphism $f\:Y\rarrow X$, the functor $f_!\:\F_Y^m
\rarrow\F_X^m$ is right adjoint to the functor
$f^*\:\F_X^m\rarrow\F_Y^m$.
 These adjunctions are induced by the similar adjunctions of functors
between the categories $\Et_X^m$ and $\Et_Y^m$.

 For any quasi-finite morphism $Y\rarrow X$, the functors $f_!$
and~$f^*$ between the exact categories $\F_X^m$ and $\F_Y^m$ satisfy
the projection formula with respect to the tensor products over~$\Z/m$.
 These functors also commute with each other in any base change
situation in which they are all defined.

\begin{lem}  \label{ext-adjunction}
 For any pair of adjoint exact functors between exact categories
$\E'\rarrow\E''$ and\/ $\E''\rarrow\E'$ the induced pair of
triangulated functors between the derived categories $\D^b(\E')$
and $\D^b(\E'')$ are naturally adjoint.
 The same assertion holds for unbounded derived categories.
\end{lem}

\begin{proof}
 This follows from a general result about adjunctions and quotients
for triangulated categories.
 For any pair of adjoint functors between triangulated categories
$\mathcal H'$ and $\mathcal H''$ taking their triangulated
subcategories $\mathcal A'$ and $\mathcal A''$ into each other,
the induced functors between the quotient categories $\mathcal H'/
\mathcal A'$ and $\mathcal H''/\mathcal A''$ are adjoint.
\end{proof}

\Section{Relative Motives}  \label{relative-motives-secn}

 For any quasi-finite morphism $f\:Y\rarrow X$ of varieties over $K$,
the \emph{cohomological relative motive of\/ $Y$ over $X$ with
compact supports}, defined as $\M_\cc^m(Y/X)=f_!(\Z/m)$, is an object
of the exact subcategory $\E_X^m$ in $\F_X^m$.
 The relative motive $\M_\cc^m(Y/X)\in\E_X^m$ is contravariantly
functorial with respect to finite morphisms of varieties $Y$
quasi-finite over $X$ and covariantly functorial with respect to
\'etale\footnote{
 In fact, the motive $\M_\cc^m(Y/X)\in\E_X^m$ is even
covariantly functorial with respect to flat morphisms $g\:Z\rarrow Y$
of varieties quasi-finite over $X$, since for any flat
quasi-finite morphism $g$ there is a natural map
$g_!\,\Z/m\rarrow\Z/m$ of \'etale sheaves over~$Y$.
 However, when the schemes $Y$ and $Z$ are not necessarily reduced,
this map depends on their non-reduced structures
(i.~e., it is \emph{not} determined by the induced morphism of
the maximal reduced closed subschemes).
 When $g$~is a finite flat morphism, the composition
$\M_\cc^m(Y/X)\rarrow\M_\cc^m(Z/X)\rarrow\M_\cc^m(Y/X)$ is
the multiplication with the degree of the morphism~$g$, defined
as the rank of the locally free sheaf $g_*\mathcal O_Z$ over~$Y$.
}
morphisms of varieties $Y$ quasi-finite over~$X$.
 For any quasi-finite morphism $Y\rarrow X$ and any closed subvariety
$Z\subset Y$ there is a natural short exact sequence
\begin{equation} \label{compact-supports-open-closed}
 0\lrarrow\M_\cc^m((Y\backslash Z)/X)\lrarrow \M_\cc^m(Y/X)
 \lrarrow\M_\cc^m(Z/X)\lrarrow 0
\end{equation}
in $\E_X^m\subset\F_X^m$.

 For any quasi-finite morphism of smooth varieties $f\:Y\rarrow X$,
the \emph{homological relative motive of\/ $Y$ over $X$} is
the object
$$
 \M_h^m(Y/X)=\M_\cc^m(Y/X)(\dim Y/X)[2\dim Y/X]
$$
in the derived category $\D^b(\F_X^m)$, where $\dim Y/X = \dim Y
- \dim X$ is the relative dimension.
 For any quasi-finite morphism of smooth varieties $Y\rarrow X$
and a smooth closed subvariety $Z\subset Y$ there is a natural
distinguished triangle
\begin{equation} \label{homological-open-closed}
 \M_h^m((Y\backslash Z)/X)\rarrow\M_h^m(Y/X)\rarrow
 \M_h^m(Z/X)(c)[2c]\rarrow\M_h^m((Y\backslash Z)/X)[1]
\end{equation}
in $\D^b(\F_X^m)$, where $c=-\dim Z/Y$ is the codimension of $Z$
in~$Y$.

 If $Y'\rarrow X$ and $Y''\rarrow X$ are two quasi-finite morphisms
of algebraic varieties and $Y=Y'\times_X Y''$, then
$\M_\cc^m(Y)=\M_\cc^m(Y')\ot\M_\cc^m(Y'')$ in $\E_X^m\subset\F_X^m$.
 If $Y'\rarrow X$ and $Y''\rarrow X$ are two quasi-finite morphisms
of smooth varieties with transversal singularities, i.~e.,
the variety $Y=Y'\times_X Y''$ is smooth and $\dim Y/X = \dim Y'/X
+ \dim Y''/X$, then $\M_h^m(Y)=\M_h^m(Y')\ot\M_h^m(Y'')$ 
in $\D^b(\F_X^m)$. 

 The basic results justifying the definition of the relative
homological motive $\M_h^m(Y/X)$ will be obtained below in
Section~\ref{homol-motives-secn}.

\begin{lem} \label{finite-etale-over-locally-closed}
 For any algebraic variety $X$, the exact category $\E_X^m$ is
generated by the objects $\M_\cc^m(Y/X)$, where $Y$ runs over all
varieties finite and \'etale over smooth locally closed subvarieties
of $X$, using the (iterated) operation of passage to an extension.
 The exact category $\F_X^m$ is generated by the objects
$\M_\cc^m(Y/X)(j)$, with $Y$ as above and $j\in\Z$, in
the same way.
\end{lem}

\begin{proof}
 Given an object $N\in\E_X^m$, consider a stratification of $X$ by
smooth locally closed subvarieties such that the restriction of
$N$ to each stratum is a lisse \'etale sheaf.
 Then notice that whenever for a lisse \'etale sheaf $M$ of
$\Z/m$\+modules on a smooth connected variety $U$ the corresponding
module over the absolute Galois group of the generic point of $U$
is a permutational module with coefficients in $\Z/m$,
the \'etale sheaf $M$ is the direct image of $\Z/m$ from a finite
\'etale morphism into~$U$.
 This is so because the maps of \'etale fundamental groups induced
by open embeddings of connected smooth varieties are surjective
(since an open subvariety of a connected smooth covering variety is
connected).
 Now it remains to use the exact sequences of \'etale sheaves
related to the extension by zero from open subvarieties.
 The second assertion follows from the first one in the obvious way.
\end{proof}

\begin{lem}  \label{generated-by-etale}
 The exact category $\E_X^m$ is generated by the objects
$\M_\cc^m(Y/X)$, where $Y$ runs over all varieties
\'etale over $X$, using the operations of passage to the cokernel
of an admissible monomorphism and (iterated) extension.
 Furthermore, any object of $\E_X^m$ is the target of an admissible
epimorphism whose source is the object $\M_\cc^m(Y/X)$,
where $Y$ is a variety \'etale over~$X$.
\end{lem}

\begin{proof}
 One deduces the first assertion from
Lemma~\ref{finite-etale-over-locally-closed} by showing that any
\'etale morphism into a locally closed subvariety $Z\subset X$
factors through an \'etale morphism into $X$, perhaps after $Z$
is replaced by its dense open subvariety.

 Passing to the local ring of a generic point of $Z$ in $X$, it
suffices to check that any \'etale morphism to the closed point of
a local scheme can be extended to an \'etale morphism to the whole
local scheme.
 This can be done by choosing a primitive element in a separable
field extension of the residue field and lifting the coefficients
of its irreducible equation to the local ring in an arbitrary way.
 Then it remains to use the exact
triple~\eqref{compact-supports-open-closed} and
the Noetherian induction.

 To prove the second assertion, one can identify the absolute Galois
group of the residue field of a scheme point of $X$ with the \'etale
fundamental group of the spectum of the Henselization of its local
ring and present the latter as a filtered projective limit of
varieties \'etale over~$X$.
\end{proof}

\begin{lem}  \label{generated-by-finite}
 The exact category $\E_X^m$ is generated by the objects
$\M_\cc^m(Y/X)$, where $Y$ runs over all normal varieties finite
over $X$, using the operations of passage to the kernel of
an admissible epimorphism and (iterated) extension.
\end{lem}

\begin{proof}
 Use Lemma~\ref{finite-etale-over-locally-closed} and the fact
that any quasi-finite morphism of algebraic varieties is
the composition of an open embedding and a finite morphism
(Grothendieck's form of Zariski's main
theorem~\cite[Th\'eor\`eme~8.12.6]{EGAIV3}) together with
the fact that the normalization is a finite morphism.
\end{proof}

 Lemmas~\ref{generated-by-etale}--\ref{generated-by-finite} allow
to compute (at least, in principle) the $\Z/m$\+modules $\Ext$
in the exact categories $\E_X^m$ and $\F_X^m$ in terms of
the $\Z/m$\+modules $\Ext$ between the Tate motives $\Z/m(j)$ in
the categories $\E_Y^m$ and $\F_Y^m$ over some other varieties~$Y$,
using the adjunctions of the functors $f_!$ and~$f_*$.
 The problem is that the varieties $Y$ may turn out to be singular,
even if the original variety $X$ was smooth.
 On the other hand, it always suffices to know
$\Ext^*_{\F_Y^m}(\Z/m,\Z/m(j))$ for normal varieties~$Y$.
 These observations will be used in the proofs in
Sections~\ref{embedding-secn}\+-\ref{applications}.

\Section{Nisnevich Topology: Distinguished Pairs and Points}

 The aim of this section and the next one is to prepare
ground for the proof of the motivic cohomology comparison theorems
in Section~\ref{main-theorem-section}.

 It is explained in Appendix that one can construct a complex of
$\Z/m$\+modules $C^\bu_\E(M,N)$ with good functorial properties
computing the modules $\Ext_\E^*(M,N)$ for any two objects
$M$ and $N$ in a $\Z/m$\+linear exact category~$\E$.
 In this section we establish some properties of the complexes
computing $\Ext$ in the exact categories $\F_X^m$ with respect
to the Nisnevich topology on the category of algebraic varieties~$X$.

 Let $X$ be a variety over $K$ and let $M$ and $N$ be two
objects in the exact category~$\F_X^m$.
 To any morphism of varieties $f\:Y\rarrow X$ we assign
the complex of $\Z/m$\+modules $C^\bu_{M,N}(Y) =
C^\bu_{\F_Y^m}(f^*M,f^*N)$.
 The map $Y\longmapsto C^\bu_{M,N}(Y)$ is a complex of presheaves
on the category of algebraic varieties mapping into~$X$.

 A \emph{distinguished pair of morphisms} (for the Nisnevich
topology) \cite[Definition~12.5]{MVW} is a pair of morphisms
of varieties $U\rarrow Y$ and $Z\rarrow Y$ such that
$U\rarrow Y$ is an open embedding and $Z\rarrow Y$ is an \'etale
morphism which is an isomorphism over $Y\backslash U$, i.e.,
$(Y\backslash U)\times_Y Z\simeq Y\backslash U$.
 
\begin{lem}  \label{nisnevich-distinguished-pairs-lemma}
 Let $\upsilon\:U\rarrow Y$ and\/ $g\:Z\rarrow Y$ be a distinguished
pair of morphisms of varieties mapping into~$X$.
 Then the total complex of the bicomplex with three rows
$$
 C^\bu_{M,N}(Y)\lrarrow C^\bu_{M,N}(U)\oplus C^\bu_{M,N}(Z)\lrarrow
 C^\bu_{M,N}(U\times_Y Z)
$$
is acyclic.
\end{lem}

\begin{proof}
 For any \'etale morphism $g\:Z\rarrow Y$ of varieties mapping into $X$,
the composition of the natural morphisms of complexes
$$
 C^\bu_{\F_Y^m}(g_!\,g^*f^*M\;f^*N) \lrarrow
 C^\bu_{\F_Z^m}(g^*g_!\,g^*f^*M\;g^*f^*N) \lrarrow
 C^\bu_{\F_Z^m}(g^*f^*M\;g^*f^*N)
$$
is a quasi-isomorphism.
 Indeed, passing to the cohomology turns this composition into
the adjunction isomorphism for $\Ext$, related to the pair of
adjoint exact functors $g_!$ and~$g^*$ between the exact categories
$\F_Y^m$ and $\F_Z^m$ and the pair of objects $g^*f^*M\in\F_Z^m$
and $f^*N\in\F_Y^m$ (see Lemma~\ref{ext-adjunction}).
 Hence the bicomplex we are interested in is row-wise
quasi-isomorphic to the bicomplex
\begin{multline*}
 C^\bu_{\F_Y^m}(f^*M\;f^*N)\lrarrow
 C^\bu_{\F_Y^m}(\upsilon_!\,\upsilon^*f^*M\;f^*N) \\
 \oplus C^\bu_{\F_Y^m}(g_!\,g^*f^*M\;f^*N)\lrarrow
 C^\bu_{\F_Y^m}(h_!\,h^*f^*M\;f^*N),
\end{multline*}
where $h$ denotes the morphism $U\times_Y Z\rarrow Y$.
 Exactness of the total complex of the latter bicomplex follows
from exactness of the short sequence
$$
 0\lrarrow h_!\,h^*f^*M\lrarrow
 \upsilon_!\,\upsilon^*f^*M\oplus g_!\,g^*f^*M
 \lrarrow f^*M\lrarrow 0
$$
in the exact category~$\F_Y^m$, which is easy to check.
\end{proof}

 Let $H$ be the spectrum of the Henselization of the local ring
of a scheme point of an algebraic variety over~$K$.
 Just as for varieties over $K$, one can define the exact
category of \'etale sheaves $\E_H^m$ and the exact category of
filtered \'etale sheaves $\F_H^m$ for the scheme~$H$.
 For any morphism $h$ from $H$ to a variety $Y$ over $K$
there is an exact functor of inverse image
$h^*\:\F_Y^m\rarrow\F_H^m$.

 Let $M$ and $N$ be objects of the category $\F_X^m$ for some
variety $X$ over~$K$.
 Given a morphism $h\:H\rarrow X$, we denote by $C^\bu_{M,N}(H)=
C^\bu_{\F_H^m}(h^*M,h^*N)$ the corresponding complex computing
$\Ext_{\F_H^m}(h^*M,h^*N)$.

\begin{lem}  \label{nisnevich-limit-lemma}
 Let $y$ be a scheme point of an algebraic variety $Y$ mapping
into $X$, and let $h_y\:H_y\rarrow Y$ be the related morphism into $Y$
from the Henselian local scheme $H_y$ corresponding to this point.
 Then the natural map of complexes of\/ $\Z/m$\+modules
$$
 (h_y^*\,C^\bu_{M,N})(H_y)\lrarrow C^\bu_{M,N}(H_y),
$$
where $h_y^*\,C^\bu_{M,N}$ denotes the inverse image of the complex of
presheaves $C^\bu_{M,N}$ under the functor between the categories of
\'etale morphisms into $H_y$ and $Y$ related to the morphism of
schemes~$h_y$, is a quasi-isomorphism.
\end{lem}

\begin{proof}
 The complex $(h_y^*\,C^\bu_{M,N})(H_y)$ is the filtered inductive
limit of the complexes $C^\bu_{M,N}(Z)$ over all varieties $Z$,
\'etale over $Y$ and endowed with a lifting $H_y\rarrow Z$ of
the morphism $H_y\rarrow Y$.
 The scheme $H_y$ is the projective limit in the category of
schemes of the filtered projective system of varieties~$Z$.
 Moreover, the abelian category $\Et^m_{H_y}$ of constructible
\'etale sheaves of $\Z/m$\+modules over $H_y$ is equivalent to
the inductive limit of the abelian categories $\Et^m_Z$
(with respect to the functors of inverse image of \'etale sheaves).

 Furthermore, on a filtered inductive limit of exact categories
there is a natural exact category structure.
 The exact categories $\E_{H_y}^m$ and $\F_{H_y}^m$ are equivalent
to the inductive limits of the exact categories $\E_Z^m$ and
$\F_Z^m$, respectively.
 Finally, the filtered inductive limits of exact categories 
commute with the passage to the $\Ext$ groups.
\end{proof}

\begin{lem} \label{nisnevich-residue-lemma}
 Let $H$ be the spectrum of the Henselization of the local ring
of a scheme point of an algebraic variety over $K$, and let
$\iota\:\eta\rarrow H$ be the closed point of~$H$.
 Let $M$ and $N$ be objects of the category $\F_H^m$.
 Assume that the successive quotient objects $\gr_F^jM$ of
the object $M$ with respect to the filtration $F$ are locally
constant \'etale sheaves on~$H$.
 Then the natural map 
$$
 C^\bu_{\F_H}(M,N)\lrarrow C^\bu_{\F_\eta}(\iota^*M\;\iota^*N),
$$
where $\iota^*\:\F_H^m\rarrow\F_\eta^m$ denotes the inverse image
with respect to~$\iota$, is a quasi-isomorphism of complexes of\/
$\Z/m$\+modules.
\end{lem}

\begin{proof}
 The restriction of the functor~$\iota^*$ to the full exact subcategory
of $\F_H^m$ consisting of the objects with locally constant
successive quotients is an equivalence of exact categories,
since the \'etale fundamental groups of $H$ and~$\eta$ coincide.
 The embedding of exact categories $\F_\eta^m\rarrow\F_H^m$
inverse to this restriction is left adjoint to the functor~$\iota^*$.

 Essentially, this is so because finite \'etale covers of $H$
are cofinal among all \'etale covers and any morphism from
a scheme finite and \'etale over~$\eta$ to a scheme \'etale over $H$
forming a commutative diagram with~$\iota$ factors through a scheme
finite and \'etale over~$H$.
 One can consider the site $\mathfrak G$ of finite \'etale morphisms
into $H$ (with the obvious topology).
 Then the \'etale site of $H$ maps to the site $\mathfrak G$ in
the obvious way (the direction of the site map being opposite to that
of the functor).

 The category of constructible sheaves over $\mathfrak G$ is equivalent
to the category of locally constant constructible \'etale sheaves
over~$H$.
 The inverse image functor for the above map of sites is the embedding
of the locally constant constructible \'etale sheaves into
arbitrary constructible \'etale sheaves over $H$, while the direct
image is identified with the functor~$\iota^*$.
 So the former is left adjoint to the latter.
 This adjunction of exact functors between the abelian categories
$\Et_H^m$ and $\Et_\eta^m$ induces the desired adjunction of exact
functors between the exact categories $\F_H^m$ and $\F_\eta^m$.

 It remains to apply Lemma~\ref{ext-adjunction} in order to deduce
the desired isomorphism of the $\Ext$ modules.
\end{proof}

 Notice that the assertions of the above three Lemmas are equally
applicable to the complexes computing the $\Z/m$\+modules $\Ext$
in the abelian categories of constructible \'etale sheaves
$\Et_Y^m$, \ $\Et_H^m$, etc.\ in place of the exact categories
$\F_Y^m$, \ $\F_H^m$, etc., and the complexes of presheaves
formed from these complexes of $\Z/m$\+modules.

\Section{Hypercohomology and Derived Direct Image}

 In this section we prove several technical lemmas.
 Given a presheaf $P$ on the category of varieties \'etale over
a given variety $X$ or on the category of all varieties
over $K$, we denote by $P_\Nis$ the sheafification of $P$ in
the Nisnevich topology.
 For a complex of presheaves $P^\bu$, the similar notation
$P^\bu_\Nis$ stands for the complex of Nisnevich sheaves
obtained by sheafifying every term of the complex~$P^\bu$.
 Given a complex of Nisnevich sheaves $S^\bu$ with the cohomology
sheaves bounded from below, $\H_\Nis^i(X,S^\bu)$ denotes
the Nisnevich hypercohomology of $X$ with coefficients in~$S^\bu$.

\begin{lem}  \label{nisnevich-cohomology-morphism}
 Let $P^\bu$ be a complex of presheaves of\/ $\Z/m$\+modules on
the category of varieties \'etale over a fixed variety $X$ over $K$
with the cohomology presheaves $H^iP^\bu$ bounded from below.
 Then there is a natural map $H^iP^\bu(X)\rarrow
\H^i_\Nis(X,P^\bu_\Nis)$ from the cohomology modules
$H^iP^\bu(X)$ of the complex of sections $P^\bu(X)$ to
the hypercohomology modules\/ $\H^i_\Nis(X,P^\bu_\Nis)$.
\end{lem}

\begin{proof}
 Let $PreSh$ denote the abelian category of presheaves of
$\Z/m$\+modules on the category of varieties \'etale over $X$,
and let $Sh_\Nis$ be the category of Nisnevich sheaves of
$\Z/m$\+modules on the same category/site.
 Then there is the functor $\Nis\:\D^+(PreSh)\rarrow\D^+(Sh_\Nis)$
induced by the (exact) Nisnevich sheafification functor
$PreSh\rarrow Sh_\Nis$.
 The functor $\Nis$ has a right adjoint functor $J\:\D^+(Sh_\Nis)
\rarrow\D^+(PreSh$) constructed as follows.
 Given a bounded below complex of sheaves $S^\bu$, choose its
bounded below injective resolution $I^\bu$ in $Sh_\Nis$ and
consider it as a complex of presheaves; then $J(S^\bu)=I^\bu$.

 Now for any complex of presheaves $P^\bu$ there is the adjunction
morphism $P^\bu\rarrow J(P^\bu_\Nis)$ in $\D^+(PreSh)$.
 The induced morphism of complexes of sections $P^\bu(X)\rarrow
J(P^\bu_\Nis)(X)$ represents the desired cohomology map
$H^*P^\bu(X)\rarrow\H^*_\Nis(X,P^\bu_\Nis)$.
\end{proof}

\begin{lem}  \label{nisnevich-cohomology-iso}
 Let $P^\bu$ be a complex of presheaves of\/ $\Z/m$\+modules on
the category of varieties \'etale over $X$ with the cohomology
presheaves $H^iP^\bu$ bounded from below.
 Assume that for any distinguished pair $U\rarrow Y$ and $Z\rarrow Y$
of morphisms of varieties \'etale over $X$ the total complex of
the bicomplex with three rows
$$
 P^\bu(Y)\lrarrow P^\bu(U)\oplus P^\bu(Z)\lrarrow P^\bu(U\times_Y Z)
$$
is acyclic.
 Then the natural map $H^iP^\bu(X)\rarrow\H^i_\Nis(X,P^\bu_\Nis)$
is an isomorphism.
\end{lem}

\begin{proof}
 We use the notation from the proof of
Lemma~\ref{nisnevich-cohomology-morphism}.
 It suffices to show that the morphism $P^\bu\rarrow J(P^\bu_\Nis)$
is a quasi-isomorphism of complexes of presheaves whenever
the complex of presheaves $P^\bu$ satisfies the condition of Lemma.

 The total complex of the bicomplex with three rows entering into
this condition is the complex of morphisms from the complex of
presheaves $(\Z/m)_{U\times_Y Z}\rarrow(\Z/m)_U\oplus(\Z/m)_Z\rarrow
(\Z/m)_Y$ into $P^\bu$, where $(\Z/m)_Y$ is the presheaf of
$\Z/m$\+modules freely generated by the presheaf of sets represented
by~$Y$.
 The presheaves $(\Z/m)_Y$ are projective objects, so this total
complex computes also the $\Hom$ in the derived category
$\D^+(PreSh)$.
 The sheafification functor sends the above three-term complex of
presheaves to an acyclic complex of Nisnevich sheaves. 
 Hence the complex $J(P^\bu_\Nis)$, being a complex of injective
sheaves, also satisfies the same condition.

 So does the cone $Q^\bu$ of the morphism of complexes of presheaves
$P^\bu\rarrow J(P^\bu_\Nis)$.
 Besides, the complex of presheaves $Q^\bu$ is annihilated by
the sheafification functor.
 From these two properties, we will deduce that $Q^\bu$ is acyclic.
 Clearly, it suffices to assume that $Q^i=0$ for $i<0$ and prove
that the differential $d\:Q^0\rarrow Q^1$ is injective.

 Suppose $s\in Q^0(Y)$ is a section annihilated by the morphism
$Q^0\rarrow Q^1$.
 Then $s$~must be also annihilated by the map $Q^0(Y)\rarrow
Q^0_\Nis(Y)$, since otherwise the morphism of sheaves
$Q^0_\Nis\rarrow Q^1_\Nis$ wouldn't be injective and the complex
of sheaves $Q^\bu_\Nis$ wouldn't be acyclic.

 Consider a Nisnevich cover $V_\alpha\rarrow Y$ of the variety $Y$
such that the section~$s$ vanishes in the restriction to~$V_\alpha$.
 Choose a stratification of $Y$ by connected locally closed
subvarieties, over each of which $V_\alpha$ has a section.
 Assume that $s\ne 0$.
 Let us throw away closed strata from this stratification one by
one until we find an open subvariety $Y'\subset Y$ such that
$s|_{Y'}$ is still nonzero, but there is a closed stratum
$W\subset Y$ such that $s|_U=0$, where $U=Y'\backslash W$.
 There exists $\beta\in\{\alpha\}$ such that the full preimage
$W\times_Y V_\beta$ of $W$ in $V_\beta$ contains a connected
component whose morphism to $W$ is an isomorphism.
 Let $Z\subset Y'\times_Y V_\beta$ denote the complement to all
the other connected components of this full preimage.
 Then the morphisms $U\rarrow Y'$ and $Z\rarrow Y'$ form
a distinguished pair.
 
 Since $s|_U=0=s|_Z$ and $d(s)=0$, while $s|_{Y'}\ne 0$ in $Q^0(Y')$
and $Q^{-1}=0$, the section $s|_{Y'}$ represents a nonzero
cohomology class in the total complex of the bicomplex associated
with this distinguished pair and the complex of presheaves~$Q^\bu$.
 This contradiction proves that $s=0$ and the complex $Q^\bu$
is acyclic. 
\end{proof}

 Now let $\U$ be a Grothendieck universe set such that $K\in \U$.
 Given a variety $X$ over $K$, let $\Et_X^{m,\U}$ denote the abelian
category of \'etale sheaves of $\Z/m$\+modules over $X$ belonging
to~$\U$; so $\Et_X^m\subset\Et_X^{m,\U}\subset\Et_X^{m,\infty}$.

 Let $\rho\:\Et\rarrow Nis$ denote the natural map between the (big)
\'etale and Nisnevich sites of algebraic varieties over~$K$.
 Let $P\in \U$ be an \'etale sheaf of $\Z/m$\+modules on the site
of all varieties over~$K$.
 The restriction of $P$ to the (small) \'etale site of \'etale
varieties over a given variety $X$ defines an object of
$\Et_X^{m,\U}$, which we will denote also by~$P$.
 Consider the presheaf of complexes of $\Z/m$\+modules
$C^\bu_P(X) = C^\bu_{\Et_X^{m,\U}}(\Z/m,P)$ on the category of
varieties over~$K$.

\begin{lem}  \label{derived-direct-image-lemma}
 The complex of Nisnevich sheaves $(C^\bu_P)_\Nis$ on the site of
algebraic varieties over $K$ represents the derived direct image\/
$\mathbb R\rho_*(P)$ of the \'etale sheaf $P$ with respect to
the map of sites\/ $\rho\:\Et\rarrow Nis$.
\end{lem}

\begin{proof}
 Applying the canonical truncation, we can assume that the complex
of presheaves $C^\bu_P$ is concentrated in nonnegative
cohomological degrees.
 Furthermore, a section of an \'etale sheaf $P$ over a variety $X$
can be viewed as a morphism $\Z/m\rarrow P$ in $\Et^{m,\U}_X$, and
to such a morphism one can naturally assign a cocycle of degree~$0$
in the complex $C^\bu_P(X)$.
 Hence there is a natural morphism $P\rarrow C^\bu_P$ of complexes
of presheaves on the category of varieties over~$K$.

 Now let $I^\bu$ be an injective resolution of the \'etale sheaf~$P$.
 Then there is the bicomplex of presheaves $C^\bu_{I^\bu}$ on
the category of varieties over~$K$.
 The complexes of presheaves $I^\bu$ and $C^\bu_P$ map naturally
into the total complex of the bicomplex $C^\bu_{I^\bu}$.
 The map from the former is a quasi-isomorphism of complexes of
presheaves, since $\Ext^i_{\Et^{m,\U}_X}(\Z/m,I^j)=0$ for $i>0$.
 The map from the latter is also a quasi-isomorphism due to
the property of the complexes $C^\bu_P$ with respect to short exact
sequences of the objects $P\in\Et^{m,\U}_X$ (see
item~\eqref{exact-triples-ext-complex} in Appendix).

 The Nisnevich sheafification transforms quasi-isomorphisms of
complexes of pre\-sheaves into quasi-isomorphisms of complexes of
Nisnevich sheaves.
 It also does not change the complex of \'etale sheaves~$I^\bu$.
 Thus the complexes $(C^\bu_P)_\Nis$ and $I^\bu$ are connected by
a natural chain of quasi-isomorphisms of complexes of Nisnevich
sheaves.
 By the definition, the complex $I^\bu$, considered as a complex
of Nisnevich sheaves, computes $\mathbb R\rho_*(P)$.
\end{proof}

\Section{Comparison Theorem}  \label{main-theorem-section}

 The following two theorems constitute the first main result of
this paper.

\begin{thm}  \label{map-theorem}
 For any algebraic variety $X$ over a field $K$ of
characteristic not dividing~$m$, there are natural
maps of\/ $\Z/m$\+modules
\begin{equation}  \label{main-comparison-map}
 \theta_X^{m,i,j}\:\Ext_{\F_X^m}^i(\Z/m,\Z/m(j))\lrarrow
 \H^i_\Nis(X,\tau_{\le j}\mathbb R\rho_*\mu_m^{\ot j}).
\end{equation}
 The compositions of these maps with the Nisnevich hypercohomology
maps induced by the morphisms
$\tau_{\le j}\mathbb R\rho_*\mu_m^{\ot j}\rarrow
\mathbb R\rho_*\mu_m^{\ot j}$
coincide with the maps\/
$\Ext_{\F_X^m}^i(\Z/m,\Z/m(j))\allowbreak\rarrow
H^i_{\acute et}(X,\mu_m^{\ot j})$
induced by the exact forgetful functor $\F_X^m\rarrow\Et_X^m$.
\end{thm}

 The maps~\eqref{main-comparison-map} for $X=\Spec L$, where $L$ is
a field of characteristic not dividing~$m$, were discussed in
the paper~\cite{mat}.
 In this case, the right-hand side of~\eqref{main-comparison-map} is
isomorphic to the Galois cohomology group $H^i(G_L,\mu_m^{\ot j})$
when $i\le j$ and vanishes for $i>j$.
 The left-hand side of~\eqref{main-comparison-map} for $X=\Spec L$
can be also easily seen to vanish for $i>j$ 
\cite[Theorem~6.1(1)]{mat}.
 The maps $\theta_{\Spec L}^{m,i,j}$ are simply induced by
the forgetful functor from the exact category $\F_{\Spec L}^m$ of
filtered $G_L$\+modules with a restriction on the successive
quotients to the abelian category of arbitrary (discrete)
$G_L$\+modules over~$\Z/m$.

 Assuming the Beilinson--Lichtenbaum conjecture for the field~$L$,
these maps can be also described as being induced by the embedding
of the exact category $\F_{\Spec L}^m\simeq\M\A\T(L,\Z/m)$ into
the triangulated category $\D\M(L,\Z/m)$ \cite[Theorem~3.1(1)]{mat}.
 The assertion that the maps $\theta_{\Spec L}^{m,i,j}$ are
isomorphisms is (a particular case of) what was called
the \emph{silly filtration} or \emph{$K(\pi,1)$\+conjecture for
Artin--Tate motives over\/ $L$ with coefficients in\/ $\Z/m$}
in~\cite{mat}.
 Specifically, it is \cite[Conjecture~9.2]{mat} for the field denoted
by $K$ in~\cite{mat} being our field $L$ and the field denoted by $M$
being its separable closure (see~\cite[Sections~9.3 and~9.9]{mat}
for some further details).

 When the field $L$ contains a primitive $m$\+root of unity, this
conjecture is equivalent to the Koszul property of the big graded
ring of the diagonal $\Ext$ between the Artin--Tate
motives~\cite[Proposition~8.1, Theorem~9.1, and Section~9.5]{mat}.
 More precisely, the maps $\theta_{\Spec L}^{m,i,j}$ are isomorphisms
for all (finite, separable) algebraic extensions $L$ of a given
field $K\ni\sqrt[m]{1}$ if and only if the big graded ring $A$
\cite[formula~(9.4)]{mat} describing the Artin--Tate motives over~$K$
with coefficients $\Z/m$ is Koszul.

\begin{thm}  \label{isomorphism-theorem}
 Given an integer $m$ and an algebraic variety $X$, the comparison
maps~\textup{\eqref{main-comparison-map}} are isomorphisms for $X$
and all the varieties $Y$ \'etale over~$X$ if and only if they are
isomorphisms for all (the spectra of the residue fields of)
the scheme points $y\in Y$.
\end{thm}

\begin{proof}[Proof of two Theorems]
{\hbadness=1300
 Consider the complex of presheaves $C^\bu_{\Z/m,\,\Z/m(j)}(X) =
C^\bu_{\F^m_X}(\Z/m,\Z/m(j))$ on the category of varieties
over~$K$.
 There is a natural morphism from it to the complex of presheaves
$C^\bu_{\Z/m,\,\mu_m^{\ot j}}(X) = C^\bu_{\Et_X^m}(\Z/m,\mu_m^{\ot j})$
induced by the forgetful functors $\F_X^m\rarrow\Et_X^m$.
 Consider the induced morphism of complexes of Nisnevich sheaves
\begin{equation} \label{nisnevich-complexes-map}
 (C^\bu_{\Z/m,\,\Z/m(j)})_\Nis\lrarrow
 (C^\bu_{\Z/m,\,\mu_m^{\ot j}})_\Nis.
\end{equation} \par}

 By Lemmas~\ref{nisnevich-limit-lemma}\+-%
\ref{nisnevich-residue-lemma}, the map of the cohomology of the stalks
at the Henzelization of the local ring of a scheme point $y\in Y$
induced by~\eqref{nisnevich-complexes-map} is identified with
the map~$\theta_y^{m,i,j}$.
 In particular, the stalks of the left-hand side
of~\eqref{nisnevich-complexes-map} are concentrated in
the cohomological degrees~$\le j$, hence there is the induced morphism
\begin{equation} \label{nisnevich-complexes-truncated-map}
 (C^\bu_{\Z/m,\,\Z/m(j)})_\Nis\lrarrow
 \tau_{\le j}\.(C^\bu_{\Z/m,\,\mu_m^{\ot j}})_\Nis.
\end{equation}
in the derived category of Nisnevich sheaves over the site of
algebraic varieties over~$K$.

 Now consider the map of Nisnevich hypercohomology of $X$ induced
by the morphism~\eqref{nisnevich-complexes-truncated-map}.
 By Lemmas~\ref{nisnevich-distinguished-pairs-lemma}
and~\ref{nisnevich-cohomology-iso}, the $\Z/m$\+modules
$\H^i_\Nis(X,(C^\bu_{\Z/m,\,\Z/m(j)})_\Nis)$ are naturally
isomorphic to $\Ext^i_{\F_X^m}(\Z/m,\Z/m(j))$.
 By Lemma~\ref{derived-direct-image-lemma}, the complex of Nisnevich
sheaves $(C^\bu_{\Z/m,\,\mu_m^{\ot j}})_\Nis$ represents
the derived direct image $\mathbb R\rho_*\mu_m^{\ot j}$.
 This proves the assertions of both Theorems.
\end{proof}

\begin{rem}
 There is a simple way to make the equivalent assertions
of Theorem~\ref{isomorphism-theorem} automatically true by
replacing the exact categories $\F_X^m$ with somewhat larger exact
categories $\F^{\prime\,m}_X$.
 Namely, let $\E_X^{\prime\,m}$ denote the abelian category $\Et_X^m$
endowed with the exact category structure in which a short sequence
is exact if it is split exact at every scheme point of~$X$.
 In other words, we just drop the condition that the Galois
representations at stalks be permutational.
 Let $\F_X^{\prime\,m}$ be the exact category of finitely filtered
objects of $\Et_X^m$ with the exact triples defined by the condition
that the triples of successive quotients by the filtration must be
exact in $\E_X^{\prime\,m}$.
 Then our proof of Theorems~\ref{map-theorem}\+-%
\ref{isomorphism-theorem} applies to the categories $\F_X^{\prime\,m}$
as well as to the categories $\F_X^m$.
 Moreover, the analogues of the maps $\theta_y^{m,i,j}$ for
the categories $\F_y^{\prime\,m}$ are easily seen to be
isomorphisms~\cite[Example~8.3 and Remark~9.3]{mat}.
 However, Lemmas~\ref{finite-etale-over-locally-closed}\+-%
\ref{generated-by-finite}, of course, do not hold for the exact
categories $\E_X^{\prime\,m}\supset\E_X^m$ and $\F_X^{\prime\,m}
\supset\F_X^m$, nor is our embedding theorem
(Theorem~\ref{embedding-theorem} below) applicable to them.

 On the other hand, let us emphasize that it is of key importance
to our $\Ext$ computation that the exact category structure on
$\F_X^m$ or $\F_X^{\prime\, m}$ is defined by the condition of
\emph{split} exactness of the short sequences of Galois modules
of stalks at scheme points.
 Indeed, this splitting condition is the reason why the canonical
truncation appears in the right-hand side of
the formula~\eqref{main-comparison-map}, i.~e., we obtain
the modified \'etale descent rule of the Beilinson--Lichtenbaum
type, rather than the conventional \'etale descent, for the $\Ext$
spaces in our exact categories (or at least certainly for
the exact categories $\F_X^{\prime\,m}$).
 For comparison, define $\F_X^{\prime\prime\,m}$ as the exact
category of filtered objects in $\Et_X^m$ with the exact triples
being the short sequences with zero composition for which the exact
triples of successive quotients are exact in $\Et_X^m$.
 Then one has $\Ext^i_{\F_X^{\prime\prime\,m}}(\Z/m,\Z/m(j))
\simeq H^i_{\acute et}(X,\mu_m^{\ot j})$
\cite[Example~D.1]{mat}.
\end{rem}

 Yet another way to modify our definition of the exact categories
$\E_X^m$ and $\F_X^m$ is to allow the Galois representations at
the stalks of the (successive quotient) sheaves to be direct summands
of permutational representations rather than permutational
representations as such.
 An argument similar to the proof of
Lemma~\ref{finite-etale-over-locally-closed} shows that the exact
categories one obtains in this way are simply the closures of 
the exact categories $\E_X^m$ and $\F_X^m$ with respect to adjoining
the images of idempotent endomorphisms.
 Thus such a change in the definitions does not affect the
$\Z/m$\+modules $\Ext$, and only leads to the necessity to mention
the passage to the direct summands in the formulations of
Lemmas~\ref{finite-etale-over-locally-closed}\+-%
\ref{generated-by-finite}.

\begin{cor}  \label{low-degree-cor}
 Let $X$ be an algebraic variety over~$K$.
 Then
\begin{enumerate}
\renewcommand{\theenumi}{\alph{enumi}}
\item for any $j\in\Z$, the map $\theta_X^{m,i,j}$
      from~\eqref{main-comparison-map} is an isomorphism for
      $i=0$, $1$, and a monomorphism for $i=2$;
\item for any $j\le 2$, the map $\theta_X^{m,i,j}$
      is an isomorphism for all~$i\in\Z$;
\item whenever $X$ is normal and connected, the\/ $\Z/m$\+module\/
      $\Ext^i_{\F_X^m}(\Z/m,\Z/m)$ vanishes for all $i\ne0$, and is
      freely generated by the identity endomorphism when $i=0$.
\end{enumerate}
\end{cor}

\begin{proof}
 First of all, it follows from the proof of
Theorem~\ref{isomorphism-theorem} above that its assertion holds when
$j$~is fixed and $i$~is restricted to an interval $0\le i\le n$ with
a fixed $n\ge0$.
 That is, given $j\in\Z$ and $n\ge0$, the maps $\theta_Y^{m,i,j}$ are
isomorphisms for all the varieties $Y$ \'etale over $X$ and all
$0\le i\le n$ if and only if the maps $\theta_y^{m,i,j}$ are
isomorphisms for all the scheme points $y\in Y$ and the same~$i$.
 Moreover, if this is the case, then the map $\theta_Y^{m,n+1,j}$
is a monomorphism for all $Y$ \'etale over $X$ if and only if
the map $\theta_y^{m,n+1,j}$ is a monomorphism for all the scheme
points $y\in Y$.

 Now part~(a) follows from~\cite[Theorem~3.1(2)]{mat}.
 To prove part~(b), one only has to check that the map
$\theta_y^{m,2,2}$ is surjective, and this so because
the Galois cohomology module $H^2(G_{K_y},\mu_m^{\ot 2})$ is
multiplicatively generated by $H^1(G_{K_y},\mu_m)$ \cite{MS1}.
 Finally, part~(c) is~\cite[Exercise~12.32(1)]{MVW}.
\end{proof}

\Section{Weak Version of Six Operations}  \label{weak-six}

 Here we spell out the additional structures on and the properties
of the triangulated categories of motivic sheaves $\D\M(X,\Z/m)$
over algebraic varieties $X$ over $K$ that we will need in order
to construct our embeddings $\F_X^m\rarrow\D\M(X,\Z/m)$.
 Of the six operations in the conventional formalism, we will use
only three: the inverse image~$f^*$, the direct image with
compact supports~$f_!$, and the tensor product
$\otimes=\otimes_{\Z/m}$.

 So, suppose that we are given the following data.
 For any algebraic variety $X$ over $K$, there is a $\Z/m$\+linear
symmetric tensor triangulated category $\D\M(X,\Z/m)$ with the unit
object $\Z/m$ and a fixed invertible object
$\Z/m(1)\in\D\M(X,\Z/m)$.
 As usually, we set $M(j)=M\ot\Z/m(1)^{\ot j}$ for
$M\in\D\M(X,\Z/m)$.
 For any morphism of algebraic varieties $f\:Y\rarrow X$,
there is a tensor triangulated functor $f^*\:\D\M(X,\Z/m)
\rarrow\D\M(Y,\Z/m)$ taking $\Z/m(1)$ to $\Z/m(1)$, and
a triangulated functor $f_!\:\D\M(Y,\Z/m)\rarrow\D\M(X,\Z/m)$.
 The following constraints and conditions are imposed.

\begin{enumerate}
\renewcommand{\theenumi}{\roman{enumi}}
\item The assignment of the functors $f^*$ and $f_!$ to
      morphisms of varieties $f\:Y\rarrow X$ takes identity
      morphisms to identity functors, and compositions to
      compositions.
      For an \'etale morphism~$f$, the functor $f_!$ is
      left adjoint to~$f^*$.
      For a proper morphism~$f$, the functor $f_!$ is
      right adjoint to~$f^*$.
      When the morphism $f$ is a universal homeomorphism,
      the functors $f_!$ and~$f^*$ are equivalences of
      triangulated categories.
\item In a base change situation, i.~e., given morphisms of
      varieties $f\:Y\rarrow X$ and $g\:Z\rarrow X$, and
      $W=Z\times_XY$ being their Cartesian product with
      the natural morphisms $f'\:W\rarrow Z$ and
      $g'\:W\rarrow Y$, the functor $f_!$ commutes with~$g^*$.
      In other words, there is an isomorphism of
      functors $g^*f_!\simeq f'_!\,g'{}^*$.
      The compatibility of these base change isomorphisms with
      the compositions of the morphisms $f$ and~$g$ holds.
      When the morphism $f$ is \'etale or proper, the base change
      isomorphism is provided by the morphism defined in terms of
      the adjunction and the compatibilities with the compositions,
      as stated in~(i).
\item For any variety $X$ with an open subvariety $\upsilon\:
      U\rarrow X$ and its closed complement $\iota\:Z\rarrow X$,
      and any object $M\in\D\M(X,\Z/m)$, there is
      a distinguished triangle
$$
 \upsilon_!\,\upsilon^*M\lrarrow M\lrarrow \iota_!\,\iota^*M\lrarrow
 \upsilon_!\,\upsilon^*M[1].
$$
      in the triangulated category $\D\M(X,\Z/m)$.
      Here the leftmost and the middle morphisms are
      the adjunction morphisms for the open embedding~$\upsilon$
      and the closed embedding~$\iota$.
      The rightmost morphism is functorial in~$M$ and
      compatible with the inverse image functors~$f^*$ with respect
      to morphisms of varieties $f\:Y\rarrow X$.
\item For any morphism of varieties $f\:Y\rarrow X$, object
      $M\in\D\M(X,\Z/m)$, and object $N\in\D\M(Y,\Z/m)$,
      there is an isomorphism $f_!(f^*M\ot N)\simeq M\ot f_!\.N$
      in the category $\D\M(X,\Z/m)$.
      This projection formula isomorphism is functorial in~$M$ and
      $N$, and compatible with the compositions of the morphisms~$f$.
      When the morphism $f$ is \'etale or proper, the projection
      formula isomorphism is provided by the morphism defined in
      terms of the adjunction and the preservation of
      the tensor product by the functor~$f^*$.
\end{enumerate}

 Furthermore, we will need to have \'etale realization functors
acting on our triangulated categories of motivic sheaves.
 These are presumed to be tensor triangulated functors
$\Phi_X\:\D\M(X,\Z/m)\rarrow\D(\Et_X^{m,\infty})$ taking values
in the derived categories of \'etale sheaves of $\Z/m$\+modules
over the varieties~$X$.
 For the sake of generality, we allow the derived categories to be
unbounded and the \'etale sheaves to be nonconstructible.
 The following constraints and conditions are imposed.

\begin{enumerate}
\setcounter{enumi}{4} \renewcommand{\theenumi}{\roman{enumi}}
\item One has $\Phi_X(\Z/m(1))=\mu_m$ for all varieties~$X$.
\item For any morphism of varieties $f\:Y\rarrow X$, the functors
      $\Phi_X$ and $\Phi_Y$ form commutative diagrams with
      the functors $f^*$ and $f_!$ between the triangulated
      categories $\D\M(X,\Z/m)$ and $\D\M(Y,\Z/m)$, and
      the similar functors between the derived categories
      $\D(\Et_X^{m,\infty})$ and $\D(\Et_Y^{m,\infty})$.
\item The functors $\Phi$ transform the constraints (i\+-iv)
      for the triangulated categories of motivic sheaves
      into the similar constraints for the derived categories
      of \'etale sheaves~\cite[Arcata~IV.5]{SGA41/2}.
\end{enumerate}

 In addition, we will have to assume that the triangulated categories
$\D\M(X,\Z/m)$ satisfy the following formulation of
the Beilinson--Lichtenbaum conjecture (proven in~\cite{SV,GL}
and~\cite{Voev-MBK}).

\begin{enumerate}
\setcounter{enumi}{7} \renewcommand{\theenumi}{\roman{enumi}}
\item For any variety $X$, the $\Z/m$\+modules
      $\Hom_{\D\M(X,\Z/m)}(\Z/m,\Z/m(j)[i])$ vanish for
      $j<0$ and all~$i$.
\item For any smooth connected variety $X$, the $\Z/m$\+module
      $\Hom_{\D\M(X,\Z/m)}(\Z/m,\allowbreak\Z/m)$ is freely generated
      by the identity endomorphism, while the $\Z/m$\+mod\-ules
      $\Hom_{\D\M(X,\Z/m)}(\Z/m,\Z/m[i])$ vanish for all $i\ne0$.
\item For any smooth variety $X$, the morphisms
      $\Hom_{\D\M(X,\Z/m)}(\Z/m,\Z/m(j)[i])\allowbreak\rarrow
      H_{\acute et}^i(X,\mu_m^{\ot j})$ induced by
      the functor $\Phi_X$ are isomorphisms for all $i\le j$
      and monomorphisms for $i=j+1$.      
\end{enumerate}

 The above assumptions will suffice for the purposes of
Sections~\ref{resolution-secn}\+-\ref{embedding-secn}, but in
Sections~\ref{applications}\+-\ref{homol-motives-secn} we will need
a more precise version of the Beilinson--Lichtenbaum conjecture.
 Let $H$ be the spectrum of the Henselization of the local ring
of a scheme point of a smooth variety over~$K$.
 Then varieties $X$ over $K$ endowed with scheme morphisms
$H\rarrow X$ form a filtered category.
 Here is our last assumption.

\begin{enumerate}
\setcounter{enumi}{10} \renewcommand{\theenumi}{\roman{enumi}}
\item The filtered inductive limit $\varinjlim_{H\to X}
      \Hom_{\D\M(X,\Z/m)}(\Z/m,\Z/m(j)[i])$ vanishes for all $i>j$
      and any Henselian scheme $H$ as above.
\end{enumerate}

\medskip
 A notable attempt to construct the triangulated categories of
motivic sheaves with the six operations formalism was undertaken
by Cisinski and D\'eglise~\cite{CD}.
 However, these categories do not seem to have all the properties
that we need in the case of motives with finite coefficients.
 The new approach by the same authors~\cite{CD2} may be much closer
to what is needed.

\Section{Resolution of Singularities} \label{resolution-secn}

 Recall that an \emph{abstract blow-up} of a variety $X$ with
a center $Z\subset X$ is a proper birational morphism
$\widetilde X\rarrow X$ that is an isomorphism over a dense
open subvariety $U$ to which $Z=X\backslash U$ is the closed
complement~\cite{Voev-triang,MVW}.
 We will say that \emph{varieties of dimension~$\le d$ over $K$
admit resolution of singularities} if for any normal variety $X$
of dimension~$\le d$ over $K$ there exists a sequence of
abstract blow-ups $X_n\rarrow\dotsb\rarrow X_1\rarrow X$ with
normal centers such that the map $X_n\rarrow X$ factors through
a proper birational morphism onto $X$ from a smooth variety.

 Notice that all the conditions (i\+-xi) of Section~\ref{weak-six}
make sense for varieties of bounded dimension, except for
the condition~(ii), which presumes that the dimension can be
increased.
 So we will say that the condition~(ii) holds for varieties of
dimension~$\le d$ if it holds whenever the varieties $X$ and $Z$
have dimensions~$\le d$ and the morphism $f\:Y\rarrow X$ is
quasi-finite.

\begin{lem} \label{resolution-lemma}
 Assume that triangulated categories of motivic sheaves
$\D\M(X,\Z/m)$ satisfying the conditions \textup{(i\+-x)}
of Section~\textup{\ref{weak-six}} are defined for varieties $X$
of dimension $\le d$ over $K$, and that such varieties admit
resolution of singularities.  Then
\begin{enumerate}
\renewcommand{\theenumi}{\alph{enumi}}
\item for any normal connected variety $X$ of dimension~$\le d$
      over $K$, the\/ $\Z/m$\+mod\-ule\/ $\Hom_{\D\M(X,\Z/m)}
      (\Z/m,\Z/m[i])$ vanishes for all $i<0$ and $i=1$, and is
      freely generated by the identity endomorphism for $i=0$;
\item for any variety $X$ of dimension~$\le d$ over $K$
      and any $j\ge1$, the morphism of\/ $\Z/m$\+modules\/
      $\Hom_{\D\M(X,\Z/m)}(\Z/m,\Z/m(j)[i])\rarrow H^i_{\acute et}
      (X,\mu_m^{\ot j})$ induced by the \'etale realization functor\/
      $\Phi_X$ is an isomorphism for $i\le 1$ and a monomorphism
      for~$i=2$.
\end{enumerate}
\end{lem}

\begin{proof}
 Let $p\:\widetilde X\rarrow X$ be an abstract blow-up of a variety
$X$ with a closed center $\iota:Z\rarrow X$ and its open complement
$\upsilon\:U\rarrow X$.
 Let $\widetilde Z$ be the Cartesian product $Z\times_X\widetilde X$;
denote the related closed embedding by $\tilde\iota\:\widetilde Z
\rarrow\widetilde X$, the projection by $\pi\:\widetilde Z \rarrow Z$,
and the open embedding by $\tilde\upsilon\:U\rarrow\widetilde X$.
 Then we have the distinguished triangle
\begin{equation} \label{downstairs-triangle}
 \upsilon_!\,\Z/m\lrarrow\Z/m\lrarrow\iota_!\,\Z/m\lrarrow
 \upsilon_!\,\Z/m[1]
\end{equation}
in $\D\M(X,\Z/m)$ and the distinguished triangle
\begin{equation*} %\label{upstairs-triangle}
 \tilde\upsilon_!\,\Z/m\lrarrow\Z/m\lrarrow\tilde\iota_!\,\Z/m
\rarrow\tilde\upsilon_!\,\Z/m[1]
\end{equation*}
in $\D\M(\widetilde X,\Z/m)$.
 Applying the functor~$p_!$ to the latter triangle, we obtain
a distinguished triangle
\begin{equation} \label{pushdown-triangle}
 \upsilon_!\,\Z/m\lrarrow p_!\,\Z/m\lrarrow\iota_!\,\pi_!\,\Z/m\
 \lrarrow\upsilon_!\,\Z/m[1]
\end{equation}
in $\D\M(X,\Z/m)$.
 There is a natural morphism from the
triangle~\eqref{downstairs-triangle} to
the triangle~\eqref{pushdown-triangle} acting by identity on
their common first vertex.
 Thus the octahedron axiom applies, and we obtain two distinguished
triangles with a common first vertex
\begin{equation} \label{compare-triangles}
\begin{gathered} 
 M\lrarrow\Z/m\lrarrow p_!\,\Z/m\lrarrow M[1] \\
 M\lrarrow\iota_!\,\Z/m\lrarrow\iota_!\,\pi_!\,\Z/m\lrarrow M[1]
\end{gathered}
\end{equation}
in $\D\M(X,\Z/m)$.
 Twisting with~$(j)$ and applying the functor $\Hom_{\D\M(X,\Z/m)}
(\Z/m,\allowbreak{-}[*])$, we get the exact sequences
\begin{multline}  \label{motivic-p-compare}
\dotsb\rarrow\Hom_{\D\M(X,\Z/m)}(\Z/m,M(j)[i])\rarrow
\Hom_{\D\M(X,\Z/m)}(\Z/m,\Z/m(j)[i]) \\ \rarrow
\Hom_{\D\M(\widetilde X,\Z/m)}(\Z/m,\Z/m(j)[i])\rarrow
\Hom_{\D\M(X,\Z/m)}(\Z/m,M(j)[i+1])
\end{multline}
and
\begin{multline}  \label{motivic-pi-compare}
\Hom_{\D\M(Z,\Z/m)}(\Z/m\;\Z/m(j)[i-1])\rarrow
\Hom_{\D\M(\widetilde Z,\Z/m)}(\Z/m\;\Z/m(j)[i-1]) \\ \rarrow
\Hom_{\D\M(X,\Z/m)}(M,\Z/m(j)[i])\rarrow
\Hom_{\D\M(Z,\Z/m)}(\Z/m,\Z/m(j)[i])\rarrow\dotsb
\end{multline}
 Applying to the triangles~\eqref{compare-triangles}
the \'etale realization functor~$\Phi_X$, twisting the result
with $\mu_m^{\ot j}$, and passing to the \'etale cohomology,
we obtain the exact sequences
\begin{multline}  \label{etale-p-compare}
 \dotsb\lrarrow\H_{\acute et}^i(X\;\Phi_X(M)\ot\mu_m^{\ot j})
 \lrarrow H_{\acute et}^i(X,\mu_m^{\ot j}) \\ \lrarrow
 H_{\acute et}^i(\widetilde X,\mu_m^{\ot j}) \lrarrow
 \H_{\acute et}^{i+1}(X\;\Phi_X(M)\ot\mu_m^{\ot j})
\end{multline}
and
\begin{multline} \label{etale-pi-compare}
 H_{\acute et}^{i-1}(Z,\mu_m^{\ot j})\lrarrow
 H_{\acute et}^{i-1}(\widetilde Z,\mu_m^{\ot j}) \\ \lrarrow
 \H_{\acute et}^i(X\;\Phi_X(M)\ot\mu_m^{\ot j})\lrarrow
 H_{\acute et}^i(Z,\mu_m^{\ot j})\lrarrow\dotsb
\end{multline}
together with morphisms of exact sequences
from~\eqref{motivic-p-compare} to~\eqref{etale-p-compare}
and from~\eqref{motivic-pi-compare} to~\eqref{etale-pi-compare}.

\medskip
 We will argue by induction in $\dim X$.
 The case of a smooth variety~$X$ is covered by the condition~(ix)
from Section~\ref{weak-six} for part~(a) and the condition~(x)
for part~(b).

 The assertion about the vanishing of $\Hom_{\D\M(X,\Z/m)}(\Z/m,
\Z/m[i])$ for $i<0$ holds for any variety~$X$.
 One proves this using a proper birational morphism $\widetilde X
\rarrow X$ onto $X$ from a smooth variety $\widetilde X$ and the exact
sequences~(\ref{motivic-p-compare}\+-\ref{motivic-pi-compare}).
 The assertion about $\Hom_{\D\M(X,\Z/m)}(\Z/m,\Z/m)$ holds
for any connected variety~$X$.
 Once again, one proves this using an abstract blow-up of $X$ with
a smooth variety $\widetilde X$, using the fact that there is
at least one connected component of $\widetilde Z$ over
every component of~$Z$.

 The assertion about the vanishing of $\Hom_{\D\M(X,\Z/m)}
(\Z/m,\Z/m[1])$ holds for any normal variety~$X$.
 To prove this, one first shows that the morphism
$\Hom_{\D\M(X,\Z/m)}(\Z/m,\Z/m[1])\rarrow
\Hom_{\D\M(\widetilde X,\Z/m)}(\Z/m,\Z/m[1])$ is  
injective for any abstract blow-up $\widetilde X\rarrow X$
with a normal center~$Z$.
 This follows from the the exact
sequences~(\ref{motivic-p-compare}\+-\ref{motivic-pi-compare})
and the fact that there is at most one
connected component of $\widetilde Z$ over every connected
component of $Z$ (Zariski's main
theorem~\cite[Proposition~4.3.5]{EGAIII1}).
 Since a composition of such abstract blow-ups factors
through a smooth variety, the desired vanishing assertion follows.
{\hbadness=2000\par}

 To prove the assertions of part~(b), consider a proper birational
morphism $\widetilde X\rarrow X$ onto $X$ from a smooth variety
$\widetilde X$ and apply the 5-lemma to the morphisms of exact
sequences $\eqref{motivic-p-compare}\to\eqref{etale-p-compare}$
and $\eqref{motivic-pi-compare}\to\eqref{etale-pi-compare}$.
\end{proof}

 The above proof is the only argument in this paper where any kind
of resolution of singularities is used.
 With the exception of Section~\ref{homol-motives-secn}, it is also
the only argument that uses the direct image functors~$f_!$ for
morphisms~$f$ that are not necessarily quasi-finite.
 To the extent that, for a particular definition of the triangulated
categories $\D\M(X,\Z/m)$, the assertions of
Lemma~\ref{resolution-lemma} could be established differently,
neither the resolution of singularities, nor the functors $f_!$ for
any but quasi-finite morphisms~$f$ would be needed for our purposes.

\Section{Embedding Theorem} \label{embedding-secn}

 In this section we only use the conditions~(i\+-viii)
of Section~\ref{weak-six}, together with the results
of Section~\ref{resolution-secn}.
 The next theorem is the second main result of this paper.

\begin{thm}  \label{embedding-theorem}
 Assume that the triangulated categories of motivic sheaves
$\D\M(X,\allowbreak\Z/m)$ satisfying the conditions \textup{(i\+-x)}
of Section~\textup{\ref{weak-six}} are defined for varieties
$X$ of dimension $\le d$ over $K$, and that such varieties
admit resolution of singularities.

 Then for any variety $X$ of dimension~$\le d$ over $K$
there is a natural tensor fully faithful functor\/
$\Theta_X\:\F_X^m\rarrow\D\M(X,\Z/m)$.
 The image of this functor is an exact subcategory
closed under extensions in the triangulated category
$\D\M(X,\Z/m)$ in the sense
of~\textup{\cite[\textit{Section}~A.8]{mat}}, and the induced
exact category structure coincides with the exact category
structure on $\F_X^m$ defined in Section~\textup{\ref{exact-at}}.

 The functors\/ $\Theta$ form commutative diagrams with
the inverse image functors $f^*$ for morphisms~$f$ of varieties
of dimension~$\le d$ and the direct image functors~$f_!$ for
quasi-finite morphisms of varieties of dimension~$\le d$.
 The composition of the embedding $\Theta_X\:\F_X^m\rarrow
\D\M(X,\Z/m)$ with the \'etale realization functor\/
$\Phi_X\:\D\M(X,\Z/m)\rarrow\D(\Et_X^{m,\infty})$ coincides with
the composition of the exact forgetful functor $\F_X^m\rarrow\Et_X^m$
and the natural embedding $\Et_X^m\rarrow\D(\Et_X^{m,\infty})$.
\end{thm}

\begin{proof}
 Let $\M\A\T(X,\Z/m)$ denote the minimal full subcategory of
the triangulated category $\D\M(X,\Z/m)$ containing all
the objects $f_!\,\Z/m(j)$, for quasi-finite morphisms $f\:Y\rarrow X$
and $j\in\Z$, and closed under extensions.
 We will check that $\M\A\T(X,\Z/m)$ is an exact subcategory of
$\D\M(X,\Z/m)$, then refine the restriction of the \'etale
realization functor $\Phi_X$ to the full subcategory
$\M\A\T(X,\Z/m)$ so as to obtain a tensor exact functor
$\Theta_X^{-1}\:\M\A\T(X,\Z/m)\rarrow\F_X^m$, and finally show that
the functor $\Theta_X^{-1}$ is an equivalence of exact categories,
so it can be inverted, providing the desired fully faithful
embedding $\Theta_X$.

 Let $\M\A(X,\Z/m)$ denote the minimal full subcategory of
$\D\M(X,\Z/m)$ containing the objects $f_!\,\Z/m$, with $f$ 
being quasi-finite morphisms into $X$, and closed under extensions.
 Clearly, the full subcategory $\M\A(X,\Z/m)\subset\D\M(X,\Z/m)$
is also generated, using iterated extensions, by the objects
$f_!\,\Z/m$ with the morphisms $f$ being finite \'etale morphisms
onto smooth locally closed subvarieties of~$X$.

 Arguing as in the proof of Lemma~\ref{generated-by-etale}, one
shows that the subcategory $\M\A(X,\Z/m)$ is generated by
the objects $f_!\,\Z/m$, with $f$ being etale morphisms into~$X$,
using the operations of passage to a cone (when such a cone
belongs to $\M\A(X,\Z/m)$) and iterated extension.
 Let us restate the latter assertion in the following more
precise form, which we will need later. {\hbadness=1250 \par}

\begin{lem}  \label{closed-in-etale}
 For any variety $X$ of dimension~$\le d$, the full subcategory
$\M\A(X,\Z/m)\allowbreak\subset\D\M(X,\Z/m)$ is generated, using
iterated extensions, by the objects $h_!\,\Z/m$, where
$h\:Z\rarrow X$ is the restriction of an \'etale morphism
$f\:Y\rarrow X$ to a smooth closed subvariety $Z\subset Y$.
\end{lem}

\begin{proof}
 See the proof of Lemma~\ref{generated-by-etale}.
\end{proof}

 Similarly, one shows in the same way as in the proof of
Lemma~\ref{generated-by-finite} that the subcategory
$\M\A(X,\Z/m)$ is generated by the objects $g_!\,\Z/m$, with
$g$ being finite morphisms into~$X$ from normal varieties,
using the operation of passage to a cocone (when such a cocone
belongs to $\M\A(X,\Z/m)$) and iterated extension.

\begin{lem}  \label{vanishing-lemma}
 For any variety $X$ of dimension~$\le d$, and
\begin{enumerate}
\renewcommand{\theenumi}{\alph{enumi}}
\item for all $M$,
      $N\in\M\A(X,\Z/m)$, any $i\in\Z$, and $j<0$,
\item for all $M$,
      $N\in\M\A(X,\Z/m)$, any $j\in\Z$, and $i<0$
\end{enumerate}
one has $\Hom_{\D\M(X,\Z/m)}(M,N(j)[i])=0$.
\end{lem}

\begin{proof}
 As explained above, it suffices to consider the case
when $M=f_!\,\Z/m$ and $N=g_!\,\Z/m(j)$, the morphism $f\:Y\rarrow X$
being \'etale and the morphism $g\:Z\rarrow X$ being finite.
 Using the adjunction properties of the inverse and direct
images for \'etale and finite morphisms together with the base
change, we conclude that
$$
 \Hom_{\D\M(X,\Z/m)}(f_!\,\Z/m,g_!\,\Z/m(j)[i])
 \.\simeq\.\Hom_{\D\M(Y\times_XZ,\,\Z/m)}(\Z/m,\Z/m(j)[i]).
$$
 The latter group vanishes for $j<0$ by the condition~(viii)
from Section~\ref{weak-six} and for $i<0$ by
Lemma~\ref{resolution-lemma}(a\+b).
\end{proof}

 It follows from Lemma~\ref{vanishing-lemma}(b) 
that $\M\A(X,\Z/m)\subset\M\A\T(X,\Z/m)$ are two exact subcategories
of $\D\M(X,\allowbreak\Z/m)$ (see~\cite[Section~A.8]{mat}).
 They are also tensor subcategories, as $f_!\,\Z/m\otimes g_!\,\Z/m
\simeq h_!\,\Z/m$ in $\D\M(X,\Z/m)$ for any quasi-finite morphisms
$f\:Y\rarrow X$ and $g\:Z\rarrow X$ with the Cartesian product
$h\:Y\times_XZ\rarrow X$ (as one can see from the projection
formula and the base change). 
 The tensor products in these exact categories are exact functors.

 According to~\cite[Section~3.1]{mat}, it follows from
Lemma~\ref{vanishing-lemma}(a) that there is a natural finite
decreasing filtration on every object $M$ of the exact category
$\M\A\T(X,\Z/m)$ with the successive quotient objects
$\gr^jM\in\M\A(X,\Z/m)(j)$.
 All morphisms in $\M\A\T(X,\Z/m)$ preserve these filtrations,
and a short sequence with zero composition in $\M\A\T(X,\Z/m)$
is exact if and only if its short sequence of successive
quotients is exact in every component number~$j$.
 The filtration on the tensor product of any two objects of
$\M\A\T(X,\Z/m)$ is the tensor product of their filtrations.

 For any morphism $f\:Y\rarrow X$ of varieties of dimension~$\le d$
over $K$, the tensor triangulated functor $f^*\:\D\M(X,\Z/m)\rarrow
\D\M(Y,\Z/m)$ takes $\M\A(X,\Z/m)$ to $\M\A(Y,\Z/m)$ and
$\M\A\T(X,\Z/m)$ to $\M\A\T(Y,\Z/m)$, defining tensor exact
functors~$f^*$ between these exact categories.
 This follows from the base change property~(ii) from
Section~\ref{weak-six}.
 For any quasi-finite morphism $g\:Z\rarrow X$, the triangulated
functor $g_!\:\D\M(Z,\Z/m)\rarrow\D\M(X,\Z/m)$ takes $\M\A(Z,\Z/m)$
to $\M\A(X,\Z/m)$ and $\M\A\T(Z,\Z/m)$ to $\M\A\T(X,\Z/m)$,
defining exact functors~$g_!$ between these exact categories.
 The exact functors $f^*$ and $g_!$ satisfy the base change and
the projection formula, because the triangulated functors $f^*$
and $g_!$ do.

 Clearly, the \'etale realization functor $\Phi_X\:\D\M(X,\Z/m)
\rarrow\D(\Et_X^{m,\infty})$ takes $\M\A\T(X,\Z/m)$ into
$\Et_X^m$, defining a tensor exact functor between these two
exact categories (the second of which is actually abelian).
 The following lemma provides a more precise assertion.

\begin{lem}
 The functor $\Phi_X\:\M\A\T(X,\Z/m)\rarrow\Et_X^m$ takes
the the full subcategory $\M\A(X,\Z/m)\subset\M\A\T(X,\Z/m)$
into the full subcategory $\E_X^m\subset\Et_X^m$, defining
a tensor exact functor $\M\A(X,\Z/m)\rarrow\E_X^m$.
\end{lem}

\begin{proof}
 Let $h\:x\rarrow X$ be the embedding of (the spectrum of
the residue field of) a scheme point of $X$.
 Then the functor $h^*\:\M\A(X,\Z/m)\rarrow\M\A(x,\Z/m)$
forms a commutative diagram with the functor
$h^*\:\Et_X^m\rarrow\Et_x^m$ and the \'etale realization
functors $\M\A(X,\Z/m)\rarrow\Et_X^m$ and
$\M\A(x,\Z/m)\rarrow\Et_x^m$.

 By the condition~(ix) from Section~\ref{weak-six} (or
by Lemma~\ref{resolution-lemma}(a)), one has
$\Hom_{\D\M(x,\Z/m)}\allowbreak(f_!\,\Z/m,g_!\,\Z/m[1])=0$ for any
(quasi\+)finite morphisms $f\:y\rarrow x$ and $g\:z\rarrow x$.
 Hence the exact category structure on $\M\A(x,\Z/m)$ is
trivial, with all exact triples being split and all objects
isomorphic to direct sums of the objects $f_!\,\Z/m$.
 It follows that the functor $\Phi_X$ takes any object of
$\M\A(X,\Z/m)$ to an \'etale sheaf on $X$ whose stalk at $x$ is
a permutational $G_{K_x}$\+module, and any exact triple in
$\M\A(X,\Z/m)$ to an exact triple of \'etale sheaves whose
stalk at~$x$ is a split exact triple.
\end{proof}

 Applying the functor $\Phi_X$ to the natural filtration of
an object of $\M\A\T(X,\Z/m)$, we obtain a finitely filtered
object of $\Et_X^m$ with the successive quotients in
$\E_X^m\ot\mu_m^{\ot j}$, i.~e., an object of~$\F_X^m$.
 This defines the desired tensor exact functor
$\Theta_X^{-1}\:\M\A\T(X,\Z/m)\rarrow\F_X^m$.
 Since the inverse and direct image functors $f^*$ and $f_*$
on the exact categories $\M\A\T({-},\Z/m)$ preserve
the natural filtrations, the functors $\Theta^{-1}$ commute
with the inverse and direct images. {\hbadness=1050 \par}

 It remains to show that the functor $\Theta^{-1}_X$ is
an equivalence of exact categories. 
 For this purpose, we will apply the result
of~\cite[Lemma~3.2]{mat}.
 The functor $\Theta_X^{-1}$ takes the generating objects
$f_!\,\Z/m(j)$ of the exact category $\M\A\T(X,\Z/m)$, where
$f\:Y\rarrow X$ are quasi-finite morphisms into $X$, to
the generating objects $\M_\cc^m(Y/X)(j)$ of the exact
category $\F_X^m$.
 We have to show that the induced morphisms
$$
 \Ext^i_{\M\A\T(X,\Z/m)}(f_!\,\Z/m,g_!\,\Z/m(j))\lrarrow
 \Ext^i_{\F_X^m}(\M_\cc^m(Y/X),\M_\cc^m(Z/X)(j))
$$
are isomorphisms for $i=0$, $1$ and monomorphisms for $i=2$,
for all pairs of quasi-finite morphisms $f\:Y\rarrow X$ and
$g\:Z\rarrow X$.

 Arguing as in the beginning of this proof and using the 5\+lemma,
we conclude that we can assume the morphism $f$ to be \'etale
and the morphism $g$ to be finite with a normal source.
 Using the adjunctions and the base change, and the compatibility
of these with the functors~$\Theta^{-1}$, we reduce the problem
to showing that the morphisms
\begin{equation}  \label{exact-comparison}
 \Ext^i_{\M\A\T(X,\Z/m)}(\Z/m,\Z/m(j))\lrarrow
 \Ext^i_{\F_X^m}(\Z/m,\Z/m(j))
\end{equation}
are isomorphisms for $i=0$, $1$ and monomorphisms for $i=2$
for all normal varieties $X$ of dimension~$\le d$ over~$K$.
 (Indeed, a scheme \'etale over a normal scheme is normal
\cite[Corollaire~I.9.10]{SGA1}.)
 We can also assume $X$ to be connected.

 Recall that for any objects $M$, $N\in\M\A\T(X,\Z/m)$,
the natural map
\begin{equation} \label{exact-triangulated}
 \Ext^i_{\M\A\T(X,\Z/m)}(M,N)\rarrow\Hom_{\D\M(X,\Z/m)}(M,N[i])
\end{equation}
is an isomorphism for $i=0$, $1$ and a monomorphism for $i=2$
\cite[Corollary~A.17]{mat}.

 Both sides of~\eqref{exact-comparison} vanish when $j<0$.
 When $j=0$ and $i\le1$, the map~\eqref{exact-comparison}
is an isomorphism by Corollary~\ref{low-degree-cor}(c) and
Lemma~\ref{resolution-lemma}(a).
 In order to prove that $\Ext^2_{\M\A\T(X,\Z/m)}(\Z/m,\Z/m)=0$,
we will use the following lemma.

\begin{lem}  \label{Zm-injective}
 For any normal variety $X$ of dimension~$\le d$, one has\/
$\Ext^1_{\M\A(X,\Z/m)}\allowbreak(M,\Z/m)=0$ for all objects
$M\in\M\A(X,\Z/m)$.
 In other words, $\Z/m$ is an injective object of the exact
category $\M\A(X,\Z/m)$.
\end{lem}

\begin{proof}
 By Lemma~\ref{closed-in-etale}, we can assume that $M=h_!\,\Z/m$,
where $h\:Z\rarrow X$ and $Z$ is the closed complement to
an open subvariety $U$ in a variety $Y$ \'etale over~$X$.
 Let $\upsilon\:U\rarrow Y$ and $f\:Y\rarrow X$ denote
the related morphisms; then we have the exact triple
$f_!\.\upsilon_!\,\Z/m\rarrow f_!\,\Z/m\rarrow h_!\,\Z/m$ in
$\M\A(X,\Z/m)$.
 Hence the induced exact sequence
\begin{multline*}
 \Hom_{\D\M(Y,\Z/m)}(\Z/m,\Z/m)\lrarrow
 \Hom_{\D\M(U,\Z/m)}(\Z/m,\Z/m) \\ \lrarrow
 \Hom_{\D\M(X,\Z/m)}(h_!\,\Z/m,\Z/m[1])\lrarrow
 \Hom_{\D\M(Y,\Z/m)}(\Z/m,\Z/m[1]).
\end{multline*}
 By~\cite[Proposition~I.10.1]{SGA1}, the connected components
of $Y$ are irreducible.
 Since $\Hom_{\D\M(Y,\Z/m)}(\Z/m,\Z/m[1])=0$ by
Lemma~\ref{resolution-lemma}(a), and there is at most one
connected component of $U$ in every connected component of $Y$,
we are done.
\end{proof}

 Finally, when $j\ge1$, both maps
$$
 \Ext^n_{\F_X^m}(\Z/m,\Z/m(j)[i])
 \lrarrow\H_\Nis^i(X,\tau_{\le j}\R\rho_*\mu_m^{\ot j})
 \lrarrow H_{\acute et}^i(X,\mu_m^{\ot j})
$$
are isomorphisms for $i\le 1$ and monomorphisms for $i=2$
(the rightmost one obviously, and the leftmost one by
Corollary~\ref{low-degree-cor}(a)).
 So are the maps
$$
 \Ext^i_{\M\A\T(X,\Z/m)}(\Z/m,\Z/m)\lrarrow
 \Hom_{\D\M(X,\Z/m)}(\Z/m,\Z/m[i])\lrarrow H_{\acute et}^i
(X,\mu_m^{\ot j})
$$
(see Lemma~\ref{resolution-lemma}(b)).
 As the diagram is commutative (the functors $\F_X^m\rarrow
\Et_X^m$ and $\M\A\T(X,\Z/m)\rarrow\Et_X^m$ forming
a commutative diagram with the functor $\Theta_X^{-1}$),
bijectivity of the maps~\eqref{exact-comparison} for $i=0$, $1$
and their injectivity for $i=2$ follows.

\medskip
 We have shown that the functor $\Theta_X^{-1}$ is an equivalence
of exact categories, so the desired fully faithful functor
$\Theta_X\:\F_X^m\rarrow\D\M(X,\Z/m)$ is constructed.
 The functor $\Theta_X$ commutes with the inverse and direct 
image functors $f^*$ and $f_!$, since the functor $\Theta_X^{-1}$
does.
 The same applies to the compatibility with the \'etale realization
functors.
 All the assertions having been verified, the embedding
theorem is proven.
\end{proof}

\Section{Particular Cases and Applications} \label{applications}

 In this section we discuss conditions under which
the morphisms~\eqref{exact-triangulated} are isomorphisms, or
an equivalence between the derived category of an exact
category of Artin--Tate motivic sheaves and an appropriate
triangulated subcategory of the triangulated category of
motivic sheaves can be established.

 We will identify the exact category $\F_X^m$ with the exact
subcategory $\M\A\T(X,\Z/M)\allowbreak\subset\D\M(X,\Z/m)$ using
the functor $\Theta_X$ from Theorem~\ref{embedding-theorem}, and
in particular, use the notation $\M_\cc^m(Y/X)=f_!\,\Z/m$
for the corresponding objects of both categories $\F_X^m$ and
$\D\M(X,\Z/m)$ (where $f\:Y\rarrow X$ is a quasi-finite morphism).

\begin{prop} \label{triangulated-comparison}
 Assume that the triangulated categories of motivic sheaves
$\D\M(X,\Z/m)$ satisfying the conditions~\textup{(i\+-xi)} of
Section~\textup{\ref{weak-six}} are defined for varieties $X$
of dimension~$\le d$ over $K$, and that such varieties admit
resolution of singularities.
 Assume further that the maps $\theta_{\Spec L}^{m,i,j}$
from~\textup{\eqref{main-comparison-map}} are isomorphisms for
all (the residue fields $L$ of) the scheme points of varieties of
dimension $\le d$ over~$K$.
 Then the natural maps
\begin{equation}
 \Ext^i_{\F_X^m}(M,\M_\cc^m(Y/X)(j))\lrarrow
 \Hom_{\D\M(X,\Z/m)}(M,\M_\cc^m(Y/X)(j)[i])
\end{equation}
are isomorphisms for all varieties $X$ of dimension~$\le d$ over
$K$, all objects $M\in\F_X^m$, and all smooth varieties $Y$
finite over~$X$.
\end{prop}

\begin{proof}
 Since the functor $\Theta_X$ is compatible with the adjunction
isomorphisms between the groups of morphisms in $\F_X^m$ and
$\D\M(X,\Z/m)$, it suffices to check that the map
$$
 \Ext^i_{\F_Y^m}(M,\Z/m(j))\lrarrow
 \Hom_{\D\M(Y,\Z/m)}(M,\Z/m(j)[i])
$$
is an isomorphism.
 For the same reason, and using Lemma~\ref{generated-by-etale},
the question can be reduced to checking that the map
\begin{equation}  \label{smooth-comparison}
 \Ext^i_{\F_X^m}(\Z/m,\Z/m(j))\lrarrow
 \Hom_{\D\M(X,\Z/m)}(\Z/m,\Z/m(j)[i])
\end{equation}
is an isomorphism for a smooth variety $X$ of dimension~$\le d$
over~$K$.

 Both sides of~\eqref{smooth-comparison} vanish when $j<0$ or $i<0$
(see Lemma~\ref{vanishing-lemma}).
 When $i\le j$, both sides map isomorphically to
$H_{\acute et}^i(X,\mu_m^{\ot j})$ (see
Theorem~\ref{isomorphism-theorem} and the condition~(x) from
Section~\ref{weak-six}) and the diagram is commutative,
so \eqref{smooth-comparison}~is an isomorphism.

 The case $0\le j<i$ is dealt with using the condition~(xi).
 We argue by induction in~$i$ for a fixed~$j$.
 First we show that the map~\eqref{smooth-comparison}
is injective.
 Let $x$ be an element in $\Ext^i_{\F_X^m}(\Z/m,\Z/m(j))$ that dies
in $\Hom_{\D\M(X,\Z/m)}(\Z/m,\Z/m(j)[i])$.
 As explained in the proof of Theorems~\ref{map-theorem}\+-%
\ref{isomorphism-theorem}, there exists a Nisnevich cover
$U_\alpha\rarrow X$ such that the element $x$ vanishes in
the restriction to~$U_\alpha$.

 Arguing as in the proof of Lemma~\ref{nisnevich-cohomology-iso},
we can assume that there exists a distinguished pair of morphisms
$\upsilon\:U\rarrow X$ and $g\:Z\rarrow X$ in the Nisnevich topology
such that $x$~vanishes in the restriction to $U$ and~$Z$.
 Let $h$~denote the morphism $W=U\times_XZ\rarrow X$.
 Consider the exact triple
$$
 0\lrarrow h_!\,\Z/m\lrarrow \upsilon_!\,\Z/m\oplus g_!\,\Z/m\lrarrow
 \Z/m\lrarrow0
$$
in $\F_X^m$.
 From the related long exact sequence of $\Ext^*_{\F_X^m}({-},\Z/m(j))$
we see that the element $x$ comes from an element
$w\in\Ext_{\F_W^m}^{i-1}(\Z/m,\Z/m(j))$.
 In the similar long exact sequence of $\Hom_{\D\M(X,\Z/m)}({-},
\Z/m(j)[*])$, the image of the element $w$ in
$\Hom_{\D\M(W,\Z/m)}(\Z/m\;\Z/m(j)[i-1])$ comes from an element
$\zeta\in\Hom_{\D\M(U,\Z/m)}(\Z/m\;\allowbreak\Z/m(j)[i-1])\oplus
\Hom_{\D\M(Z,\Z/m)}(\Z/m\;\Z/m(j)[i-1])$.
 By the induction assumption, the latter comes from an element
$z\in\Ext_{\F_U^m}^{i-1}(\Z/m,\Z/m(j))\oplus
\Ext_{\F_Z^m}^{i-1}(\Z/m,\Z/m(j))$.
 Continuing this diagram chase, one easily concludes that $x=0$.

 Now we can check that the map~\eqref{smooth-comparison} is
surjective.
 Let $\xi$ be an element in $\Hom_{\D\M(X,\Z/m)}(\Z/m,\Z/m(j)[i])$
which does not come from $\Ext^i_{\F_X^m}(\Z/m,\Z/m)$.
 By the condition~(xi), there exists a Nisnevich cover $U_\alpha
\rarrow X$ such that the element $\xi$ vanishes in the restriction
to~$U_\alpha$.
 Arguing as in the proof of Lemma~\ref{nisnevich-cohomology-iso}
again, we can assume that there exists a distinguished pair of
morphisms $\upsilon$ and~$g$ such that the elements $\upsilon^*\xi$
and $g^*\xi$ come from $\Ext_{\F_U^m}^i(\Z/m,\Z/m(j))$ and
$\Ext_{\F_Z^m}^i(\Z/m,\Z/m(j))$, respectively.
 Let $z$ be the corresponding element in the direct sum of
the latter two groups.
 In view of the injectivity assertion that we have proven,
the image of $z$ vanishes in $\Ext_{\F_W^m}^i(\Z/m,\Z/m(j))$,
so the element $z$ comes from a certain element in
$\Ext_{\F_X^m}^i(\Z/m,\Z/m(j))$, etc.
\end{proof}

\begin{cor} \label{curves-cor}
 Assume that the triangulated categories of motivic sheaves
$\D\M(X,\allowbreak\Z/m)$ satisfying the conditions \textup{(i\+-xi)}
of Section~\textup{\ref{weak-six}} are defined for varieties
$X$ of dimension $\le 1$ over $K$.
 Assume further that the maps $\theta_{\Spec L}^{m,i,j}$
from~\textup{\eqref{main-comparison-map}} are isomorphisms
for all (the residue fields $L$ of) the scheme points of
varieties of dimension~$\le 1$ over~$K$.
 Then the natural maps
\begin{equation}
 \Ext^i_{\F_X^m}(M,N)\lrarrow\Hom_{\D\M(X,\Z/m)}(M,N[i])
\end{equation}
\textup{(}see also~\textup{\eqref{exact-triangulated})} are
isomorphisms for all varieties $X$ of dimension~$\le 1$ over $K$
and all objects $M$, $N\in\F_X^m$.
\end{cor}

\begin{proof}
 Notice that varieties of dimension~$\le 1$ over a perfect
field $K$ always admit resolution of singularities (in the sense
of Section~\ref{resolution-secn}) by definition.
 By Lemma~\ref{generated-by-finite}, the exact category $\F_X^m$
is generated, using the operations of passage to the kernel of
an admissible epimorphism and iterated extension, by objects
$\M_\cc^m(Y/X)(j)$, where $Y$ runs over all smooth varieties
finite over~$X$ (since normal curves are smooth).
 So it remains to use Proposition~\ref{triangulated-comparison}.
\end{proof}

\begin{cor}  \label{low-weight-curves}
 Assume that the triangulated categories of motivic sheaves
$\D\M(X,\allowbreak\Z/m)$ satisfying the conditions \textup{(i\+-xi)}
of Section~\textup{\ref{weak-six}} are defined for varieties
$X$ of dimension $\le 1$ over $K$.
 Then the natural maps
\begin{equation}
 \Ext^i_{\F_X^m}(M,N(j))\lrarrow\Hom_{\D\M(X,\Z/m)}(M,N(j)[i])
\end{equation}
are isomorphisms for all varieties $X$ of dimension~$\le 1$ over $K$,
all objects $M$, $N\in\E_X^m$, and all $j\le2$.
\end{cor}

\begin{proof}
 It suffices to notice that the proof of
Proposition~\ref{triangulated-comparison}, and hence also of
Corollary~\ref{curves-cor}, holds for every fixed value of~$j$.
 Then one applies Corollary~\ref{low-degree-cor}(b).
\end{proof}

 Notice that the assertion of the latter Corollary is certainly
\emph{not} true for surfaces.
 The problem arises when $j=0$.
 Indeed, let $Y$ be a normal surface with a point singularity
such that the exceptional fiber $Z$ of its resolution $\widetilde Y$
is a self-intersecting projective line (see Introduction).
 Using the exact
sequences~(\ref{motivic-p-compare}\+-\ref{motivic-pi-compare}),
one easily computes $\Hom_{\D\M(Z,\Z/m)}(\Z/m,\Z/m[1])=\Z/m$ and
$\Hom_{\D\M(Y,\Z/m)}(\Z/m,\Z/m[2])=\Z/m$.
 However, $\Ext^2_{\F_Y^m}(\Z/m,\Z/m)=0$ by
Corollary~\ref{low-degree-cor}(c) or Lemma~\ref{Zm-injective}.

 Furthermore, given a smooth surface $X$ over an algebraically
closed field $K$, one can find a normal surface $Y$ finite over
$X$ having at least one singularity of the above type.
 Then $\Ext_{\F_X^m}^2(\Z/m,\M_\cc^m(Y/X))\simeq
\Ext_{\F_Y^m}^2(\Z/m,\Z/m)=0$ and
$\Hom_{\D\M(X,\Z/m)}(\Z/m,\M_\cc^m(Y/X)[2])\simeq
\Hom_{\D\M(Y,\Z/m)}(\Z/m,\Z/m[2])\ne0$.

\medskip
 In order to extend the functor $\Theta_X$ to a triangulated functor
$\D^b(\F_X^m)\rarrow\D\M(X,\Z/m)$, one needs some additional
structure on the triangulated categories of motivic sheaves
$\D\M(X,\Z/m)$.
 E.~g., it would suffice if these triangulated categories had
algebraic origin.

\begin{cor}  \label{realization-cor}
 Assume that the triangulated categories of motivic sheaves
$\D\M(X,\allowbreak\Z/m)$ satisfying the conditions \textup{(i\+-x)}
of Section~\textup{\ref{weak-six}} are defined for varieties
$X$ of dimension $\le d$ over $K$,
and that such varieties admit resolution of singularities.
 Assume additionally that the triangulated categories\/
$\D\M\F(X,\Z/m)$ of two-step filtrations, or morphisms
in the triangulated categories\/ $\D\M(X,\Z/m)$
\textup{\cite[\textit{Appendix}~D]{mat}} are defined.
 Then
\begin{enumerate} 
\renewcommand{\theenumi}{\alph{enumi}} 
\item there exist triangulated functors\/
\begin{equation} \label{realization}
 \widetilde\Theta_X\: \D^b(\F_X^m)\lrarrow\D\M(X,\Z/m)
\end{equation}
      extending the fully faithful functors $\Theta_X\:\F_X^m\rarrow
      \D\M(X,\Z/m)$;
\item if the functors of inverse image~$f^*$ for morphisms of
      varieties~$f$ and direct image with compact supports~$f_!$
      for quasi-finite morphisms~$f$ act on the triangulated
      categories $\D\M\F(X,\Z/m)$ in a way compatible with
      the structure functors $\kappa_0$, $\kappa_1$, $w$ between
      the triangulated categories $\D\M(X,\Z/m)$ and\/
      $\D\M\F(X,\Z/m)$, then the triangulated
      functors~\textup{\eqref{realization}}
      commute with the functors $f^*$ and~$f_!$.
\end{enumerate}
\end{cor}

\begin{proof}
 Part~(a) is provided by~\cite[Theorem~D.4]{mat}, and part~(b)
follows from the construction in the proof of that Theorem.
\end{proof}

 Let $\F_{X;\,[0,2]}^m$ denote the full exact subcategory of
$\F_X^m$ consisting of all the objects concentrated in
the filtration components $0$, $1$, and~$2$ (i.~e., the iterated
extensions of the objects of $\E_X^m$, $\E_X^m(1)$, and
$\E_X^m(2)$).
 Denote by $\D\M\A\T(X,\Z/m)\subset\D\M(X,\Z/m)$ the full triangulated
subcategory generated by $\M\A\T(X,\Z/m)$ in $\D\M(X,\Z/m)$, and by
$\D\M\A\T_{[0,2]}(X,\Z/m)$ the full triangulated subcategory
generated by $\M\A\T_{[0,2]}=\F_{X;\,[0,2]}^m$ in
$\D\M(X,\Z/m)$.

\begin{cor}  \label{curves-realization}
 Assume that the triangulated categories of motivic sheaves
$\D\M(X,\allowbreak\Z/m)$ satisfying the conditions \textup{(i\+-xi)}
of Section~\textup{\ref{weak-six}} are defined for varieties
$X$ of dimension $\le 1$ over $K$.
 Assume additionally that triangulated categories
$\D\M\F(X,\Z/m)$ are defined.
 Then
\begin{enumerate} 
\renewcommand{\theenumi}{\alph{enumi}} 
\item the functors~\textup{\eqref{realization}} are fully faithful
      in restriction to $\D^b(\F_{X;\,[0,2]}^m)$ and induce
      equivalences of triangulated categories $\D^b(\F_{X;\,[0,2]}^m)
      \simeq\D\M\A\T_{[0,2]}(X,\Z/m)$;
\item if the maps $\theta_{\Spec L}^{m,i,j}$
      from~\textup{\eqref{main-comparison-map}}
      are isomorphisms for all (the residue fields $L$ of)
      the scheme points of varieties of dimension~$\le 1$ over $K$,
      then the functors~\textup{\eqref{realization}} are fully faithful
      and induce equivalences of triangulated categories
      $\D^b(\F_X^m)\simeq\D\M\A\T(X,\Z/m)$.
\end{enumerate}
\end{cor}

\begin{proof}
 Follows from Corollaries~\ref{curves-cor}\+-\ref{realization-cor}.
\end{proof}

\begin{lem} \label{cohom-dim-1}
 Let $L$ be a field whose $l$\+cohomological dimension does not
exceed\/~$1$ for every prime number~$l$ dividing~$m$.
 Then the maps $\theta_{\Spec L}^{m,i,j}$
from~\textup{\eqref{main-comparison-map}} are isomorphisms.
\end{lem}

\begin{proof}
 Our cohomological dimension assumption means that the right-hand
side of the map~\eqref{main-comparison-map} vanishes for $X$ \'etale over
$\Spec L$ whenever $i\ge2$.
 Applying Corollary~\ref{low-degree-cor}(a) and using adjunctions as
in the first paragraph of the proof of
Corollary~\ref{triangulated-comparison} (or just applying
directly~\cite[Theorem~3.1(2)]{mat}), one concludes that
$\Ext_{\F_{\Spec L}^m}^2(\M_\cc^m(\Spec M/\Spec L)\;\M_\cc^m(\Spec N/\Spec L)(j))
=0$ for all finite separable field extensions $M$ and $N/L$ and
all~$j$.
 So the exact category $\F_{\Spec L}^m$ has homological dimension~$\le 1$
and the left-hand side of~\eqref{main-comparison-map} also vanishes
for $X=\Spec L$ and $i\ge2$.
 By Corollary~\ref{low-degree-cor}(a) again, the assertion of Lemma
follows. 
\end{proof}

\begin{cor} \label{curves-over-closed-field}
 The conclusions of Corollaries~\textup{\ref{curves-cor}}
and~\textup{\ref{curves-realization}(b)} hold whenever the field~$K$
is algebraically closed (provided that the relevant motivic
triangulated categories are defined).
\end{cor}

\begin{proof}
 Follows from Lemma~\ref{cohom-dim-1} and Tsen's theorem.
\end{proof}

\begin{rem}
 One would like to extend the result of
Corollary~\ref{curves-over-closed-field} to further cases, e.~g.,
to finite fields~$K$.
 A step in this direction was made in the paper~\cite{num}, where
the $K(\pi,1)$\+conjecture was proven for Tate motives
(see~\cite[Sections~2 and~9.1]{mat} and~\cite[Section~5]{num})
and for Artin--Tate motives related to a fixed cyclic extension
of prime degree (see \cite[Section~9.8]{mat} and~\cite[Section~6]{num})
over a one-dimensional global field containing certain roots of unity.
\end{rem}

\Section{Homological Motives}  \label{homol-motives-secn}

 Here we discuss the properties of the relative homological motives
$\M_h^m(Y/X)\in\D^b(\F_X^m)$, which were defined in
Section~\ref{relative-motives-secn} for quasi-finite morphisms of
smooth varieties $Y\rarrow X$.

 For this purpose we will need to have a fourth operation $f^!$
of the ``six operations'' formalism defined on our triangulated
categories of motivic sheaves $\D\M(X,\Z/m)$.
 So we assume, in addition to the assumptions of
Section~\ref{weak-six}, the following.

\begin{enumerate}
\renewcommand{\theenumi}{\roman{enumi}}
\setcounter{enumi}{11}
\item For any morphism $f\:Y\rarrow X$ of varieties over $K$,
      the functor $f_!\:\D\M(Y,\allowbreak\Z/m)\rarrow\D\M(X,\Z/m)$
      admits a right adjoint functor $f^!\:\D\M(X,\Z/m)
      \rarrow\D\M(Y,\Z/m)$.
      Whenever the morphism $f$ is smooth of relative dimension~$n$,
      there is a functorial isomorphism $f^!(M)\simeq f^*(M)(n)[2n]$
      for all objects $M\in\D\M(X,\Z/m)$.
\end{enumerate}

\begin{prop}  \label{homological-motive-conjecture}
 Assume that the triangulated categories of motivic sheaves
$\D\M(X,\Z/m)$ satisfying the conditions~\textup{(i\+-xii)} are
defined for varieties $X$ of dimension~$\le d$ over $K$, and
that such varieties admit resolution of singularities.
 Assume further that the maps $\theta_L^{m,i,j}$
from~\eqref{main-comparison-map} are isomorphisms for all
(the residue fields $L$ of) the scheme points of varieties of
dimension~$\le d$ over~$K$.

 Then the relative motive $\M_h^m(Y/X)\in\D^b(\F_X^m)$ is
covariantly functorial with respect to arbitrary morphisms of
smooth varieties $Y$ quasi-finite over a fixed smooth variety
$X$ of dimension~$\le d$ over~$K$.
 For any quasi-finite morphism of smooth varieties $Y\rarrow X$
of dimension~$\le d$ there is a natural isomorphism of\/
$\Z/m$\+modules
\begin{equation}  \label{smooth-homol-adjunction}
 \Hom_{\D^b(\F_X^m)}(\M_h^m(Y/X),\.\Z/m(j)[i])\.\simeq\.
 \Hom_{\D^b(\F_Y^m)}(\Z/m,\.\Z/m(j)[i]).
\end{equation} 
 Both $\Z/m$\+modules are also naturally isomorphic to
the motivic cohomology module $\Hom_{\D\M(Y,\Z/m)}(\Z/m,\Z/m(j)[i])$.
\end{prop}

\begin{proof}
 No triangulated categories $\D\M(X,\Z/m)$ are mentioned
in the formulations of the functoriality claim and
the isomorphism~\eqref{smooth-homol-adjunction}, but they
are used in the following proof, so we need them in
the assumptions.
 Notice that the motive $\M_h^m(Y/X)\in\D^b(\F_X^m)$ is a shift of
an object of $\F_X^m$ (or, at worst, is naturally presented as
a direct sum of such shifts), so an object of the category
$\D\M(X,\Z/m)$ can be associated with the motive $\M_h^m(Y/X)$
using the functor $\Theta_X\:\F_X^m\rarrow\D\M(X,\Z/m)$.

 The object so obtained is naturally isomorphic to
$f_!\,f^!\,\Z/m\in\D\M(X,\Z/m)$.
 One can see this by applying the condition~(xii) to the smooth
morphisms $X\rarrow\Spec K$ and $Y\rarrow\Spec K$ together with
the object $M=\Z/m\in\D\M(\Spec K,\.\Z/m)$, and using
the compatibility of the functors~$f^!$ with the compositions of
the morphisms~$f$ (which follows from the similar compatibility
of the functors~$f_!$).

 Both sides of the desired isomorphism~\eqref{smooth-homol-adjunction}
are (at worst, direct sums of) certain $\Ext$ groups in the exact
categories $\F_X^m$ and $\F_Y^m$, respectively.
 In our assumptions, Proposition~\ref{triangulated-comparison}
allows to identify these groups with the related $\Hom$ groups
in the categories $\D\M(X,\Z/m)$ and $\D\M(Y,\Z/m)$.
 In order to deduce~\eqref{smooth-homol-adjunction},
it remains to use the adjunction isomorphism
$$
 \Hom_{\D\M(X,\Z/m)}(f_!\,f^!\,\Z/m,\.\Z/m(j)[i])\.\simeq\.
 \Hom_{\D\M(Y,\Z/m)}(f^!\,\Z/m,\.f^!\.\Z/m(j)[i])
$$
together with the isomorphism
$$
 \Hom_{\D\M(Y,\Z/m)}(f^!\,\Z/m,\.f^!\.\Z/m(j)[i])\.\simeq\.
 \Hom_{\D\M(Y,\Z/m)}(\Z/m,\.\Z/m),
$$
which again holds in view of the condition~(xii).

 In particular, given a quasi-finite morphism of smooth varieties
$g\:Z\rarrow Y$, there is a natural morphism $\M_h^m(Z/Y)\rarrow
\Z/m$ in $\D^b(\F_Y^m)$ corresponding to the identity morphism
$\Z/m\rarrow\Z/m$ in $\D^b(\F_Z^m)$.
 Applying the direct image functor~$f_!$ with respect to
a quasi-finite morphism of smooth varieties $f\:Y\rarrow X$
together with the twist and the shift, we obtain the desired
functoriality map
\begin{equation}  \label{smooth-homol-functoriality}
 g_*\:\M_h^m(Z/X)\lrarrow\M_h^m(Y/X)
\end{equation}
in $\D^b(\F_X^m)$.
\end{proof}

 Alternatively, given a morphism $g\:Z\rarrow Y$ of smooth varieties
quasi-finite over $X$ with $\dim Y - \dim Z\le 1$, one can define
the map~\eqref{smooth-homol-functoriality} in terms of
the adjunction  morphism $f_!\,g_!\,g^!\,f^!\,\Z/m\rarrow
f_!\,f^!\,\Z/m$ in $\D\M(X,\Z/m)$, using the fact that
$\F_X^m$ is a full exact subcategory closed under extensions in
$\D\M(X,\Z/m)$ (see~\eqref{exact-triangulated}).
 The latter construction does not even depend on any
$K(\pi,1)$\+conjectures.

\appendix \bigskip
\section*{Appendix. Complexes Computing Ext in Exact Categories}
\medskip \setcounter{section}{1}  \setcounter{thm}{0}

 Let $k$ be a commutative ring and $\E$ be a $k$\+linear small
exact category.
 To any two objects $M$ and $N\in\E$ we would like to assign
a homotopy projective complex of projective $k$\+modules
$C^\bu_\E(M,N)$ with the following properties.
\begin{enumerate}
\item To any morphism $f\:M\rarrow N$ in $\E$ a cocycle
$c(f)\in C_\E^0(M,N)$ is assigned.
 One has $c(f)=0$ whenever $f=0$.
\item For any three objects $L$, $M$, $N\in\E$ there is
a composition map $C_\E^\bu(M,N)\ot_k C_\E^\bu(L,M)
\rarrow C_\E^\bu(L,N)$.
 These multiplications are associative, the elements $c(\id_M)$
are unit objects for them, and the composition $c(f)c(g)$
is equal to $c(fg)$ whenever $f$ and $g$ are composable.
\end{enumerate}

 In particular, it follows from (1\+-2) that $C^\bu_\E(M,N)$
are contravariantly functorial in $M$ and covariantly
functorial in~$N$.
 However, the maps $f\longmapsto c(f)$ are \emph{not} necessarily
compatible with the addition of morphisms~$f$ or their
multiplication by constants from~$k$.
 So the bifunctor $(M,N)\longmapsto C^\bu_\E(M,N)$ may be neither
biadditive nor $k^*$\+biequivariant.
\begin{enumerate}
\setcounter{enumi}{2}
\item For any $k$\+linear exact functor $\gamma\:\E\rarrow\F$
and all objects $M$, $N\in\E$ there are morphisms of complexes
$C^\bu_\E(M,N)\rarrow C^\bu_\F(\gamma(M),\gamma(N))$, compatible
with the structures in~(1\+-2) in the obvious sense.
\item For any $M$, $N\in\E$, there are isomorphisms between
the cohomology $k$\+modules $H^iC^\bu_\E(M,N)$ and the Yoneda
$\Ext$ $k$\+modules $\Ext^i_\E(M,N)$, compatible with all
the structures in~(1\+-3).
\item\label{exact-triples-ext-complex}
 The three-term sequences of complexes $C_\E^\bu(M,N)$
corresponding to exact triples in either of the arguments $M$
and $N$, considered as bicomplexes with three rows, have
acyclic total complexes.
\end{enumerate}

\begin{prop}
 Complexes $C^\bu_\E(M,N)$ with the above-listed properties exist.
\end{prop}

\begin{proof}
 Let us start with the simple case when $k$~is a field.
 Then it suffices to consider the DG\+category of (bounded or
unbounded) complexes over $\E$ and take its Drinfeld
localization~\cite{Dr} over~$k$ with respect to the full
DG\+subcategory of acyclic complexes.
 The localization is a $k$\+linear DG\+category $D_\E^\bu$ endowed
with a natural $k$\+linear functor into it from the  category~$\E$.
 The complexes of morphisms between the images of objects
$M$, $N\in\E$ in the DG\+category $D_\E^\bu$ provide the desired
complexes $C^\bu_\E(M,N)$.

 In the general case, there is a problem that the Drinfeld
localization is only defined for DG\+categories whose complexes
of morphisms are homotopy flat complexes of $k$\+modules.
 In order to deal with it, we will use the following lemma.

\begin{lem}
 There exists a functor $P$ from the category of complexes of
$k$\+modules to the category of homotopy projective complexes
of projective $k$\+modules with the following properties.
 The functor $P$ is endowed with natural transformations
$A^\bu\rarrow P(A^\bu)\rarrow A^\bu$ for any complex of
$k$\+modules~$A^\bu$.
 The map $P(A^\bu)\rarrow A^\bu$ is a quasi-isomorphism of complexes
of $k$\+modules.
 The map $A^\bu\rarrow P(A^\bu)$ is a map of graded sets taking zero
elements to zero elements and commuting with the differentials
(but not necessarily preserving either the addition of
elements or their multiplication by constants from~$k$).
 The composition $A^\bu\rarrow P(A^\bu)\rarrow A^\bu$ is
the identity map.

 Furthermore, $P$ is a pseudotensor functor, i.~e., for any
complexes of $k$\+modules $A^\bu$ and $B^\bu$ there is a natural
morphism of complexes of $k$\+modules $P(A^\bu)\ot_k P(B^\bu)\rarrow
P(A^\bu\ot_k B^\bu)$ compatible with the associativity and (graded)
commutativity constraints in the tensor category of complexes.
 There is a morphism of complexes of $k$\+modules $k\rarrow P(k)$
that is a pseudounit for the pseudotensor functor~$P$.
 So, in particular, $P$ transforms DG\+algebras over~$k$
to DG\+algebras and $k$\+linear DG\+categories to DG\+categories.
 The maps of section $A^\bu\rarrow P(A^\bu)$ and projection
$P(A^\bu)\rarrow A^\bu$ are multiplicative (for homogeneous elements)
with respect to this pseudotensor structure.
 The image of the element\/ $1\in k$ under the pseudounit map
$k\rarrow P(k)$ coincides with the image of the same element
under the section map $k\rarrow P(k)$.
\end{lem}

 Notice that an alternative way to assign to a DG\+category $D^\bu$
a DG\+category $F(D^\bu)$ with homotopy $k$\+projective complexes
of morphisms and a quasi-isomorphism $F(D^\bu)\rarrow D^\bu$ is to
construct a functorial cofibrant resolution of $D^\bu$ in
the model category of DG\+categories over~$k$ \cite{Tab}.
 This is not difficult to do, and in addition one can have
a natural section $D^\bu\rarrow F(D^\bu)$.
 However, this section will not be multiplicative.
 This problem is solved by the construction of the functor~$P$,
which produces from a DG\+category $D^\bu$ the DG\+category
$P(D^\bu)$ whose complexes of morphisms are cofibrant as complexes
of $k$\+modules only.

\begin{proof}[Proof of Lemma]
 To convince ourselves that such a construction is at all possible,
let us first discuss the simple case when the ring~$k$ contains
a field~$f$.
 Then one can set $P(A^\bu)$ to be the (direct sum total complex of)
the (reduced or nonreduced) bar-construction of $A^\bu$ over~$k$
relative to~$f$.
 The desired pseudotensor structure is given by the operation of
shuffle product on the bar-complexes~\cite[Proposition~1.1 of
Chapter~3]{PP}.
 The section map in this case is even additive and $f$\+linear
(but not $k$\+linear).

 In the general case, our functor $P$ takes a $k$\+module $M$
to its resolution whose degree-zero term is the $k$\+module
freely generated by all the elements of $M$, modulo the only
relation that the generator corresponding to the zero element
is equal to zero.
 This construction is iterated to obtain the whole resolution.
 To a complex of $k$\+modules $A^\bu$, the functor $P$ assigns
the total complex of the bicomplex formed by the above-described
resolutions of the terms of the complex~$A^\bu$.
 The total complex is constructed by taking infinite direct sums
along the diagonals.
 The key step is to construct the pseudotensor structure on
this functor~$P$.

 Formally, for any complex of $k$\+modules $A^\bu$ we construct
the terms of the bicomplex $P_i^j(A^\bu)$ by induction in~$i$.
 The $k$\+module $P_0^j(A^\bu)$ is generated by the symbols $[a]$,
where $a\in A^j$, with the relation $[0]=0$.
 The map $\pi\:P_0^j(A^\bu)\rarrow A^\bu$ takes $[a]$ to~$a$.
 For every $i>0$, we define the $k$\+module $P_i^j(A^\bu)$
together with the differential $\d\:P_i^j(A^\bu)\rarrow
P_{i-1}^j(A^\bu)$ as follows.
 The $k$\+module $P_i^j(A^\bu)$ is generated by the symbols
$\langle p\rangle$, where $p\in P_{i-1}^j(A^\bu)$ and $\d(p)=0$
(if $i>1$) or $\pi(p)=0$ (if $i=1$), with the relation
$\langle 0\rangle=0$.
 The differential $\d$ is defined by the rules $\d\langle p
\rangle = p$ and $\d[a]=0$.

 The degree $|p|$ of an element $p\in P_i^j(A^\bu)$ is set to
be equal to $j-i$ (notice that $\langle\ \rangle$ is
consequently an operation of degree~$-1$).
 The differential~$d$ on $P(A^\bu)$ is defined by induction and
in terms of the differential~$d$ on $A$ by the rules
$d[a] = [da]$ and $d\langle p\rangle = \langle -dp\rangle$.
 The total differential on $P(A^\bu)$ is $\d+d$.
 Finally, the pseudotensor structure is defined by induction
by the rules
\begin{gather*}
 \langle p\rangle\times\langle q\rangle =
 \langle\. p\times\langle q\rangle - (-1)^{|p|}
 \langle p\rangle \times q \.\rangle, \\
 \langle p\rangle\times[a] = \langle\. p\times[a]\.\rangle,
 \quad [a]\times\langle q\rangle =
 \langle\. (-1)^{|a|}[a]\times q\.\rangle,
\end{gather*}
and $[a]\times[b]=[a\ot b]$.
 The projection map $\pi\:P(A^\bu)\rarrow A^\bu$ is extended
from $P_0^*(A^\bu)$ to the whole of $P(A^\bu)$ by the rule
$\pi(\langle p \rangle) = 0$.
 The section map $s\:A^\bu\rarrow P(A^\bu)$ is defined on
homogeneous elements $a\in A$ by the rule $s(a)=[a]$.
 The pseudounit map $e\:k\rarrow P(k)$ is defined by
the rule $e(1)=[1]$.
\end{proof}

 Now, given a $k$\+linear exact category $\E$, we apply
the pseudotensor functor $P$ to (every complex of morphisms
in) the $k$\+linear DG\+category of complexes over~$\E$.
 To the $k$\+linear DG\+category so obtained we apply, in turn,
the Drinfeld localization construction with respect to
the full DG\+subcategory consisting of the acyclic complexes
over~$\E$.
 The resulting DG\+category $D^\bu_\E$ comes together with
a (nonadditive) functor $S\:\E\rarrow D^\bu_\E$, which
is defined in terms of the section map for the pseudotensor
functor~$P$.
 Finally, we set $C_\E^\bu(M,N)=\Hom_{D^\bu_\E}(SM,SN)$.
\end{proof}

\end{document}